\documentclass[12pt]{article}
\usepackage[utf8]{inputenc}
\usepackage[T1]{fontenc}
\usepackage{mathtools,amssymb,amsthm,amsmath}
\usepackage{hyperref}
\usepackage{enumitem}
\usepackage{esint, color}
\usepackage{lmodern}
\usepackage{geometry} 
\usepackage{amsbsy} 
\usepackage{mathrsfs}

\usepackage{tikz}
\usetikzlibrary{patterns}

\def\mcA{\mathcal{A}}
\def\Sphere{\mathbb{S}}

\def\bS{{\mathbb S}}
\newcommand \eps{\epsilon}
\newcommand \SN{{\mathbb{S}^{N-1}}}

\newcommand \N{\mathbb{N}}
\newcommand \R{\mathbb{R}}
\newcommand \RR{\mathbb{R}}

\newcommand \loc{\mathrm{loc}}

\newcounter{marnote}

\numberwithin{equation}{section}

\def\be{\begin{equation}}
\def\ee{\end{equation}}

\newtheorem{theorem}{Theorem} 
\newtheorem{corollary}[theorem]{Corollary} 
\newtheorem{lemma} [theorem]{Lemma} 
\newtheorem{remark} [theorem]{Remark} 
\newtheorem{proposition} [theorem]{Proposition} 
 
\theoremstyle{definition}

\title
{Minimality of vortex solutions to Ginzburg--Landau type systems for gradient fields in the unit ball in dimension $N\geq 4$
}

\begin{document}

\author{Radu Ignat\thanks{Institut de Math\'ematiques de Toulouse, UMR 5219, Universit\'e de Toulouse, CNRS, UPS
IMT, F-31062 Toulouse Cedex 9, France. Email: Radu.Ignat@math.univ-toulouse.fr}~, Mickael Nahon\thanks{Laboratoire Jean Kuntzmann, UMR 5224, Universit\'e Grenoble Alpes, 700 avenue centrale, 38401 Saint Martin d'H\`eres. Email: mickael.nahon@univ-grenoble-alpes.fr}~ and Luc Nguyen\thanks{Mathematical Institute and St Edmund Hall, University of Oxford, Andrew Wiles Building, Radcliffe Observatory Quarter, Woodstock Road, Oxford OX2 6GG, United Kingdom. Email: luc.nguyen@maths.ox.ac.uk}}

\date{}

\maketitle

\begin{abstract}
We prove that the degree-one vortex solution is the unique minimizer for the Ginzburg--Landau functional for gradient fields (that is, the Aviles--Giga model) in the unit ball $B^N$ in dimension $N \geq 4$ and with respect to its boundary value. A similar result is also proved for $\mathbb{S}^N$-valued maps in the theory of micromagnetics. Two methods are presented. The first method is an extension of the analogous technique previously used to treat the unconstrained Ginzburg--Landau functional in dimension $N \geq 7$. The second method uses a symmetrization procedure for gradient fields such that the $L^2$-norm is invariant while the $L^p$-norm, $2 < p < \infty$, and the $H^1$-norm are lowered.  


\end{abstract}

\tableofcontents

\section{Introduction}

Let $B^N$ be the unit ball in $\RR^N$. Consider the Ginzburg--Landau (GL) functional 
\[
E^{GL}_\eps [U] = \int_{B^N} \Big[\frac{1}{2} |\nabla U|^2 + \frac{1}{2\eps^2} W(1 - |U|^2)\Big]\,dx, 
\]
where $\eps > 0$, $W(t)=\frac{t^2}2$ and $U$ belongs to the set
\[
\mcA^{GL} = \{U \in H^1(B^N,\RR^N) : U(x) = x \text{ on } \partial B^N\}.
\]
The functional $E^{GL}_\eps$ has a unique radially symmetric critical point in $\mcA^{GL}$ of the form 
 \begin{equation}
U_\eps(x) = f_\eps(r) \frac{x}{r} \in \mcA^{GL},  \quad r = |x|,
	\label{eq_urad}
\end{equation}
where the profile $f_\eps$ is the unique solution to the ODE (see e.g. \cite{Hervex2, ODE_INSZ})
\begin{equation}\label{eq_f}
\begin{cases}
-f_\eps''(r)-\frac{N-1}{r}f_\eps'(r)+\frac{N-1}{r^2}f_\eps(r)=\frac1{\eps^2}f_\eps(r)W'(1-f_\eps(r)^2),\\
f_\eps(0)=0,\ f_\eps(1)=1.
\end{cases}
\end{equation}
Moreover $f_\eps > 0$ and $f_\eps' > 0$ in $(0,1)$.

The map $U_\eps$ in \eqref{eq_urad}, called the \emph{($\RR^N$-valued) Ginzburg--Landau vortex solution} of topological degree one, can be considered as a regularization of the singular harmonic map $n: B^N \rightarrow \Sphere^{N-1}$ given by $n(x) = \frac{x}{|x|}$ for every $x\in B^N$, which is the unique minimizing $\bS^{N-1}$-valued harmonic map for $N \geq 3$ with respect to the boundary condition $n(x)=x$ on $\partial B^N$ (see Brezis, Coron and Lieb \cite{BrezisCoronLieb} and Lin \cite{Lin-CR87}). The question about the minimality of $U_\eps$ for any $\eps > 0$ was raised in dimension $N = 2$ in Bethuel, Brezis and H\'elein \cite[Problem 10, page 139]{BBH_book}, and in higher dimensions in Brezis \cite[Section 2]{Brezis99_PSPM}. It is not hard to see that, when $\eps$ is sufficiently large, $E^{GL}_\eps$ is strictly convex and so $U_\eps$ is the unique bounded critical point of $E^{GL}_\eps$ in $\mcA^{GL}$ for every $N\geq 2$ (see e.g. \cite{BBH_book} or \cite[Remark 3.3]{INSZ_AnnENS}). In dimension $N = 2$, Pacard and Rivi\`ere showed in \cite{Pacard_Riviere} that, for small $\eps > 0$, $U_\eps$ is the unique critical point of $E^{GL}_\eps$ in $\mcA^{GL}$; however, whether $u_\eps$ is the unique minimizer of $E_\eps^{GL}$ for all $\eps > 0$ remains an open question. In dimensions $N \geq 7$, this question was answered positively in recent works of Ignat, Nguyen, Slastikov and Zarnescu \cite{INSZ18_CRAS, INSZ_AnnENS}: $U_\eps$ is the unique minimizer of $E^{GL}_\eps$ in $\mcA^{GL}$ \emph{for every} $\eps > 0$. It is not known whether $U_\eps$ minimizes $E^{GL}_\eps$ in $\mcA^{GL}$ in dimensions $3 \leq N \leq 6$ when $\eps$ is small. However, it is known that for every $\eps>0$, $U_\eps$ is a local minimizer of $E^{GL}_\eps$ in $\mcA^{GL}$ -- for dimension $N=2$, see Mironescu \cite{Mironescu-radial} and also Lieb and Loss \cite{LiebLoss95-JEDP}; for dimension $3 \leq N \leq 6$, see Ignat and Nguyen \cite{IN24-IHP}. 

We note also that, when the domain is the whole space $\R^N$ instead, the minimality (in the sense of De Giorgi) of the vortex solution is available: see Mironescu \cite{Mironescu_symmetry}, Millot and Pisante \cite{Mil-Pis} and Pisante \cite{Pisante11-JFA}. See also \cite{Pino-Felmer-Kow, GravejatPacherieSmets22, Gustafson, OvchinSigal95} for studies on stability issues.

The main aim of this paper is to show that in dimensions $4\leq N\leq 6$ and \emph{for every} $\eps > 0$, $U_\eps$ is  the unique minimizer of $E^{GL}_\eps$  relative to the set of \textbf{gradient field configurations} in $\mcA^{GL}$ (this is often referred to as the Aviles--Giga model).

\subsection{The Aviles--Giga model}

Consider a general non-negative convex $C^2$ potential $W:(-\infty, 1]\to [0, \infty)$ such that $W(0)=0$ and for every $\eps>0$, the Ginzburg--Landau energy $E^{GL}_\eps(U)$ restricted to gradient fields 
$$U=\nabla u \in H^1(B^N,\R^N) \quad \textrm{ such that } \quad U|_{\partial B^N}=Id.$$ 
Within a suitable rescaling (i.e., $\eps E^{GL}_\eps(\nabla u)$), this is the so-called  Aviles--Giga model (introduced with the standard potential $W(t)=t^2/2$).

Note that the ($\RR^N$-valued) Ginzburg--Landau vortex solution $U_\eps$ introduced in \eqref{eq_urad} is a gradient field $U_\eps=\nabla u_\eps$ for some radial function $u_\eps=u_\eps(r)$ determined (up to a constant) by $u'_\eps=f_\eps$ in $(0,1)$ where $f_\eps$ is the unique solution in \eqref{eq_f}. 

We prove the following result: \begin{theorem}
\label{thm:AG}
Assume that $4\leq N\leq 6$ and $W:(-\infty, 1]\to [0, \infty)$ is a $C^2$ non-negative convex function such that $W(0)=0$. For every $\eps>0$, the radially symmetric vortex solution $U_\eps$ in \eqref{eq_urad} is the unique minimizer of $E^{GL}_\eps$ over the set of gradient fields
$\{U=\nabla u \in \mcA^{GL}\}$.
\end{theorem}

Note that the above result holds in dimension $N\geq 7$ as a consequence of \cite{INSZ18_CRAS, INSZ_AnnENS}. We expect the result holds also in dimension $N \in \{2,3\}$. We mention here the work Lorent \cite{Lorent12-ESIAM, Lorent14-AnnPisa} and Lamy and Marconi \cite{LamyMarconi23-SNS} on stability of the vortex solution in dimension $N = 2$ and in the limit $\eps \rightarrow 0$ (for the Aviles--Giga model as well as other micromagnetic models).

\subsection{The $\Sphere^N$-valued Ginzburg--Landau  model}

We consider the following model:
\[
E^{MM}_\eta [M] = \int_{B^N} \Big[\frac{1}{2} |\nabla M|^2 + \frac{1}{2\eta^2} \tilde W(M_{N+1}^2)\Big]\,dx
\]
where $\eta > 0$ and $M=(\nabla m, M_{N+1})$ is a unit-length vector field that is a gradient field in the first $N$ components belonging to
\[
\mcA^{MM} = \{M=(\nabla m, M_{N+1})\in H^1(B^N,\Sphere^N)\, :\, M(x)=(x,0) \textrm{ on } \partial B^N\}.
\]
The non-negative potential $\tilde W: [0,\infty) \rightarrow [0,\infty)$ is a $C^2$ convex function such that  
$\tilde W(0)=0$.

This model comes from micromagnetics where the order parameter $M$ stands for the magnetization in ferromagnetic materials (see \cite{Gioia-James})\footnote{In dimension $N=2$, $E_\eta^{MM}$ is the reduced energy functional in a certain thin-film ferromagnetic regime (see e.g. \cite[Section 4.5]{DKMO02-CPAM} or \cite[Section~7]{Ignat09-SEDP}) where, after a rotation by $\frac{\pi}{2}$ in the first two components of $M$, the condition $\nabla \times (M_1,M_2)=0$ is imposed in the space of admissible configurations in $\mcA^{MM}$. }, and also the Oseen-Frank theory for nematic liquid crystals (see \cite{An-Sh}).
Considering radially symmetric critical points of $E^{MM}_\eta$ in $\mcA^{MM}$, one is led to 
\begin{equation}
M_{\eta}(x) =( \tilde f_\eta(r) \frac{x}{r}, g_\eta(r))\in \mcA^{MM}
	\label{Eq:metaform}
\end{equation}
where the radial profiles $\tilde f_\eta$ and $g_\eta$ satisfy 
\begin{equation}
\label{f2g2}
\tilde f_\eta^2+g_\eta^2=1 \quad \textrm{in} \quad (0,1),
\end{equation}
and the system of ODEs:
\begin{align}
-\tilde f_{\eta}'' - \frac{N-1}{r} \tilde f_{\eta}' + \frac{N-1}{r^2} \tilde f_{\eta}
	&= \lambda(r) \tilde f_{\eta}  \quad \textrm{in} \quad (0,1),
	\label{Eq:MM-fee}\\
-g_{\eta}'' - \frac{N-1}{r} g_{\eta}' 
	&= -\frac1{\eta^2}\tilde W'(g_\eta^2) g_\eta + \lambda(r) g_{\eta}  \quad \textrm{in} \quad (0,1),
	\label{Eq:MM-gee}\\
\tilde f_{\eta}(1) &= 1 \text{ and } g_{\eta}(1) = 0,
	\label{Eq:MM-feegeeBC}
\end{align}
where 
\begin{equation}
\lambda(r)=(\tilde f_{\eta}')^2+\frac{N-1}{r^2}\tilde f_{\eta}^2+(g_{\eta}')^2+\frac1{\eta^2}\tilde W'(g_\eta^2)g_{\eta}^2
\label{Eq:lamDef}
\end{equation} is the Lagrange multiplier due to the unit length constraint in $\mcA^{MM}$. Note that indeed the vortex solution $M_\eta$ in \eqref{Eq:metaform} is of the form $M_\eta=(\nabla m_\eta, M_{\eta, N+1}) \in \mcA^{MM}$ for some radial function $m_\eta=m_\eta(r)$ determined (up to a constant) by $m'_\eta=\tilde f_\eta$ in $(0,1)$.

As proved in \cite{IN24-IHP}, the solutions to \eqref{Eq:metaform}--\eqref{Eq:MM-feegeeBC} satisfy the dichotomy: either $\tilde f_\eta(0) = 0$ or $\tilde f_\eta(0) = 1$. Furthermore, in the latter case, it holds that $N \geq 3$ and $(\tilde f_\eta=1, g_\eta=0)$ in $(0,1)$, which corresponds to the \emph{equator map}
$$
\bar M(x):=(\frac{x}{r}, 0).
$$
In dimension $N \geq 7$, $\bar M$ is the unique minimizing harmonic map from $B^N$ into $\Sphere^N$ in $H^1(B^N,\mathbb{S}^N)$ with with boundary condition $(Id,0)$ on $\partial B^N$ (J\"ager and Kaul \cite{JagerKaul83-JRAM}; see also Sandier and Shafrir \cite{SandierShafrir94-CVPDE} and \cite[Example 1.6]{INSZ_AnnENS}); so $\bar M$ is the unique minimizer of $E^{MM}_\eta$ in $\mcA^{MM}$ \emph{for every} $\eta>0$. Therefore, in the following, we focus on dimensions
$2\leq N\leq 6$
and on {\it escaping} $\Sphere^N$-valued radially symmetric vortex solutions   
$$M_{\eta}^\pm(x) =( \tilde f_\eta(r) \frac{x}{r}, \pm g_\eta(r)) \quad \textrm{with} \quad g_\eta > 0 \textrm{ in } (0,1).$$
It was proved in Hang and Lin \cite{HangLin01-ActaSin} in dimension $N = 2$ and \cite{IN24-IHP} in dimension $3 \leq N \leq 6$ that, for any $\eta > 0$,  \eqref{Eq:metaform}--\eqref{Eq:MM-feegeeBC} has a unique escaping solution $(\tilde f_\eta, g_\eta)$ with $g_\eta > 0$ and $M_\eta^\pm$ are locally minimizers for $E_\eta^{MM}$.  Moreover, $\tilde f_\eta(0) = 0$, $\tilde f_\eta > 0$, $\tilde f_\eta' > 0$ and $g_\eta' < 0$ in $(0,1)$. (See also \cite{LiMelcher18-JFA} for a related work in the context of micromagnetic skyrmions in $\R^2$.)

We prove the following result:
\begin{theorem}
\label{thm2}
Assume $4\leq N\leq 6$ and $\tilde W: [0,\infty) \rightarrow [0,\infty)$ is a $C^2$ non-negative convex function such that  
$\tilde W(0)=0$. For every $\eta>0$, $E^{MM}_{\eta}$ has exactly two minimizers over the set
$\{(\nabla m, M_{N+1})\in \mcA^{MM}\}$ and they are given by the escaping vortex solutions 
$M_{\eta}^\pm(x) = (\tilde f_{\eta}(r)\frac{x}{r}, \pm g_{\eta}(r))$ with $g_\eta > 0$ in $(0,1)$. In particular, minimizers of $E^{MM}_{\eta}$ in $\mcA^{MM}$ are radially symmetric for every $\eta > 0$.
\end{theorem}

As in the case of the Aviles--Giga model, we expect the above result holds also in dimension $N \in \{2,3\}$.

\subsection{The extended model}\label{ssec:1.3}

More generally, we consider a family of extended energy functionals $E_{\eps,\eta}$ depending on two positive parameters $\eps, \eta$ of which $E_\eps^{GL}$ and $E_\eta^{MM}$ are limiting cases when $\eta \rightarrow 0$ and $\eps \rightarrow 0$, respectively:
\begin{equation}
\label{ener}
E_{\eps,\eta}[U]= \int_{B^N} \Big[\frac{1}{2}|\nabla U|^2 + \frac{1}{2\eps^2} W(1 - |U|^2)  + \frac{1}{2\eta^2} \tilde W(U_{N+1}^2)\Big]\,dx, \quad \eps,\eta > 0,
\end{equation}
where $U=(\nabla u, U_{N+1}):B^N\to \R^{N+1}$ is a gradient field in the first $N$ components and belongs to
\[
\mcA = \{U=(\nabla u, U_{N+1}) \in H^1(B^N,\RR^{N+1}): U(x) = (x,0) \text{ on } \partial B^N\}.
\]
Here, $W: (-\infty,1] \rightarrow [0,\infty)$ and $\tilde W: [0,\infty) \rightarrow [0,\infty)$ are non-negative $C^2$ convex potentials such that $W(0) = \tilde W(0) = 0$. We point out that these imply that $W'(0) = 0$, $tW'(t) \geq 0$ in $(-\infty,1] \setminus \{0\}$, and $\tilde W'(t) \geq 0$ in $[0,\infty)$. However, we allow the possibility that  $W$ or $\tilde W$ can be zero in a neighborhood of the origin. 

Radially symmetric critical points of $E_{\eps,\eta}$ in $\mcA$ take the form
\begin{equation}
U_{\eps, \eta}=(f_{\eps,\eta}(r)\frac{x}{r},g_{\eps,\eta}(r)) \in \mcA,
	\label{Eq:feegeeH1}
\end{equation}
where $(f_{\eps,\eta},g_{\eps,\eta})$ satisfies the system of  ODEs
\begin{align}
&-f_{\eps,\eta}'' - \frac{N-1}{r} f_{\eps,\eta}' + \frac{N-1}{r^2} f_{\eps,\eta}
	= \frac{1}{\eps^2} W'(1 -  f_{\eps,\eta}^2 - g_{\eps,\eta}^2) f_{\eps,\eta}
	,\label{Eq:20III21-fee}\\
&-g_{\eps,\eta}'' - \frac{N-1}{r} g_{\eps,\eta}' 
	= \frac{1}{\eps^2} W'(1 -  f_{\eps,\eta}^2 - g_{\eps,\eta}^2) g_{\eps,\eta} - \frac{1}{\eta^2} \tilde W'(g_{\eps,\eta}^2)g_{\eps,\eta}
	,\label{Eq:20III21-gee}\\
&f_{\eps,\eta}(1) 
	= 1  \text{ and } g_{\eps,\eta}(1) = 0
	.\label{Eq:20III21-feegeeBC}
\end{align}
Note that the above implies $f_{\eps,\eta}(0) = 0$ and $g_{\eps,\eta}'(0) = 0$ (see \cite[Lemma A.5]{IN24-IHP}).  Also, note that the first $N$ components of $U_{\eps,\eta}(r)$ is a gradient field $\nabla \varphi_{\eps, \eta}$ for some radial function 
$\varphi_{\eps, \eta}(r)$ determined (up to a constant) by $\varphi_{\eps, \eta}'=f_{\eps,\eta}$ in $(0,1)$.

In dimensions $N \geq 7$, it follows from \cite{INSZ18_CRAS, INSZ_AnnENS} that the \emph{non-escaping} vortex solution
$$\bar U_\eps(x)=(f_\eps(r) \frac{x}{r},0)$$
is the unique global minimizer of $E_{\eps,\eta}$ in $\mcA$ for every $\eps>0$ and $\eta>0$. Therefore, in the following, we focus on dimensions
$2\leq N\leq 6$; we will analyse
{\it escaping} radially symmetric vortex solutions
$$U_{\eps, \eta}^\pm=(f_{\eps,\eta}(r)\frac{x}{r}, \pm g_{\eps,\eta}(r)), \quad g_{\eps,\eta}>0 \textrm{ in } (0,1).$$
It is shown by \cite{IN24-IHP} that such an escaping radially symmetric critical point $U_{\eps,\eta}$ with $g_{\eps,\eta} > 0$ exists if and only if $2 \leq N \leq 6$, $W'(1)>0$, $0 < \eps < \eps_0$ and $\eta > \eta_0(\eps)$ for some $\eps_0 \in (0,\infty)$ and a continuous non-decreasing function $\eta_0: [0,\eps_0) \rightarrow [0,\infty)$ with $\eta_0(0) = 0$. In this case, it is the unique escaping solution of \eqref{Eq:feegeeH1}--\eqref{Eq:20III21-feegeeBC} with $g_{\eps,\eta}>0$ in $(0,1)$; moreover, we have $f_{\eps,\eta}(0) = 0$, $f_{\eps,\eta}^2 + g_{\eps,\eta}^2 < 1$, $f_{\eps,\eta} > 0$, $f_{\eps,\eta}' > 0$, $g_{\eps,\eta}' < 0$ in $(0,1)$. See Section \ref{ssec:1.4} and Figure \ref{Fig1} for more information.

\bigskip

We prove the following theorem:
\begin{theorem}
\label{thm1}
Suppose $4\leq N\leq 6$ and $W: (-\infty,1] \rightarrow [0,\infty)$ and $\tilde W: [0,\infty) \rightarrow [0,\infty)$ are $C^2$ non-negative convex functions satisfying $W(0) = \tilde W(0) = 0$. For every $\eps >0, \eta > 0$, we have the following dichotomy:
\begin{itemize}
\item Either the escaping radially symmetric vortex solutions $U_{\eps, \eta}^\pm$ exist and they are the only two minimizers of $E_{\eps,\eta}$ in $\mcA$,
\item Or the escaping radially symmetric vortex solutions $U_{\eps, \eta}^\pm$ do not exist and the non-escaping vortex solution $\bar U_\eps$ is the unique minimizer of $E_{\eps,\eta}$ in $\mcA$.
\end{itemize}
In particular, minimizers of $E_{\eps,\eta}$ in $\mcA$ are always radially symmetric for every $\eps,\eta > 0$.
\end{theorem}

To complete the picture, we recall facts from \cite{IN24-IHP} on the escaping vs. non-escaping phenomena. The escaping phenomenon is related to the loss of stability of the non-escaping vortex solution $\bar U_\eps$. More precisely, consider the stability operator $\frac{\delta^2 E_{\eps,\eta}}{\delta U_{N+1}^2}$ at $\bar U_\eps$ along the $N+1$ direction:
\[
\bar T_{\eps,\eta} = -\Delta - \frac{1}{\eps^2} W'(1-f_\eps^2) + \frac{1}{\eta^2} \tilde W'(0) .
\]
The first eigenvalue of $\bar T_{\eps,\eta}$ on $H^1_0(B^N,\RR)$ takes the form $\ell(\eps)  + \frac{1}{\eta^2} \tilde W'(0) $ where $\ell(\eps)$ is the first eigenvalue of 
\begin{equation}
L_\eps := -\Delta - \frac{1}{\eps^2} W'(1-f_\eps^2).
	\label{Eq:Lepsdef}
\end{equation}
Then the 
\[
\text{ escaping vortex solutions $U_{\eps,\eta}^\pm$ with $g_{\eps,\eta} > 0$ exists if and only if $\ell(\eps)  + \frac{1}{\eta^2} \tilde W'(0) < 0$.}
\]
When $N \geq 7$ or $W'(1) = 0$, it holds always that $\ell(\eps) > 0$, hence escaping vortex solutions do not exist. When $2 \leq N \leq 6$ and $W'(1) > 0$,
\[
\text{there exists $\eps_0 > 0$ such that $\ell(\eps) > 0$ for $\eps > \eps_0$ and $\ell(\eps) < 0$ for $0 < \eps < \eps_0$.}
\]
Thus, in this case, the function $\eta_0(\eps)$ mentioned above (so that escaping vortex solutions exist if and only if $0 <\eps < \eps_0$ and $\eta > \eta_0(\eps)$) is given by
\[
\eta_0(\eps) = \sqrt{\frac{\tilde W'(0)}{|\ell(\eps)|}} \text{ for } 0 < \eps < \eps_0.
\]
In Figure \ref{Fig1}, we describe the dichotomy of escaping and non-escaping phenomena for minimizers\footnote{Recall from \cite{IN24-IHP} that solutions of \eqref{Eq:20III21-fee}-\eqref{Eq:20III21-feegeeBC} satisfying $g_{\eps,\eta} > 0$, when they exist, are minimizing for $E_{\eps,\eta}$ relatively to the set of radially symmetric configurations.} of $E_{\eps,\eta}$ in radial symmetry in dimension $2 \leq N \leq 6$. Theorem \ref{thm1} asserts that, in dimension $4 \leq N \leq 6$, this picture remains valid in the larger set $\mcA$ of gradient field configurations in the first $N$ components.

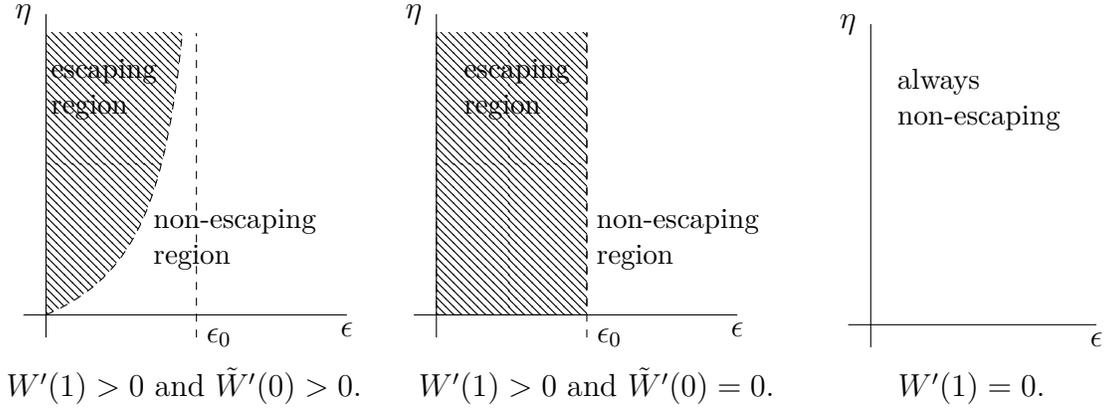
\begin{figure}[h]
\begin{center}
\caption{\small Escaping vs. Non-escaping phenomenon in dimension $2 \leq N \leq 6$. }\label{Fig1}
\begin{tabular}{ccc}
\begin{tikzpicture}
\draw (-.3,0)--(4,0);
\draw (4,-.2) node {$\eps$};
\draw (0,-.3)--(0,4);
\draw (-.3,4) node {$\eta$};
\draw[dashed] (2,-.3)--(2,3.75);
\draw (2.3,-.3) node {$\eps_0$};
\draw (2.7,1) node {\small \parbox{1in}{non-escaping region}};
\begin{scope}
\clip (0,0) rectangle (3,3.75);
\draw [domain=0:2, samples = 25, variable=\t, dashed, pattern = north west lines] plot ({\t},{\t/(2.3-\t)}) -- (0,4);
\draw (0.7,3) node {\small \parbox{.5in}{escaping region}};
\end{scope}
\end{tikzpicture}
&
\begin{tikzpicture}
\draw (-.3,0)--(4,0);
\draw (4,-.2) node {$\eps$};
\draw (0,-.3)--(0,4);
\draw (-.3,4) node {$\eta$};
\draw[dashed] (2,-.3)--(2,3.75);
\draw (2.3,-.3) node {$\eps_0$};
\draw (3.4,1) node {\small \parbox{1in}{non-escaping region}};

\begin{scope}
\clip (0,0) rectangle (3,3.75);
\draw [dashed,pattern = north west lines] (0,0) rectangle (2,4);
\draw (1,3) node {\small \parbox{.5in}{escaping region}};
\end{scope}
\end{tikzpicture}
&
\begin{tikzpicture}
\draw (-.3,0)--(3,0);
\draw (3,-.2) node {$\eps$};
\draw (0,-.3)--(0,4);
\draw (-.3,4) node {$\eta$};
\draw (1.5,3) node {\small \parbox{.9in}{always\\non-escaping}};
\end{tikzpicture}
\\
$W'(1) > 0$ and $\tilde W'(0) > 0$.
&
$W'(1) > 0$ and $\tilde W'(0) = 0$.
&
$W'(1) = 0$.
\end{tabular}
\end{center}
\end{figure}

For the case $\eta = \infty$ (that is the $\RR^{N+1}$-valued Ginzburg--Landau model), we refer the reader to the recent article Ignat and Rus \cite{IgnatRus23-arxiv}. For a similar bifurcation from non-escaping to escaping phenomenon, see Bethuel, Brezis, Coleman and H\'elein \cite{BethuelBrezisColemanHelein}.
\subsection{Ideas of the proofs}\label{ssec:1.4}

Theorems \ref{thm:AG} and \ref{thm2} will be obtained from Theorem \ref{thm1} by taking the limits $\eta\to 0$ or $\eps\to 0$, respectively. For simplicity, instead of describing the proof of Theorem \ref{thm1} (which is the main result), we explain instead the strategy of the proof in the case $\eta = 0$, i.e. Theorem~\ref{thm:AG} for the Aviles--Giga model. We will present two methods of proof. Roughly speaking, the first method follows the strategy in \cite{INSZ18_CRAS, INSZ_AnnENS} adapted to the situation of gradient field configurations, and the second method uses a symmetrization procedure. 

\subsubsection*{Method 1}

It was observed in \cite{INSZ18_CRAS, INSZ_AnnENS} that the convexity of the potential $W$ implies the inequality
\begin{equation}
E_\eps^{GL}[U_\eps + V] - E_\eps^{GL}[U_\eps] \geq \frac{1}{2} \int_{B^N} L_\eps V \cdot V\,dx =: F_\eps[V] \text{ for } V \in H_0^1(B^N,\RR^N),
	\label{Eq:ConvRed}
\end{equation}
where $L_\eps$ is the operator defined in \eqref{Eq:Lepsdef}. Recall that in dimension $N \geq 7$, the first eigenvalue $\ell(\eps)$ of $L_\eps$ is positive, $F_\eps[V] > 0$ for $V \in H_0^1(B^N,\RR^N) \setminus \{0\}$, and hence, $U_\eps$ is the unique minimizer of $E_\eps^{GL}$ in $\mcA^{GL}$. In dimension $2 \leq N \leq 6$, one has $\ell(\eps) < 0$ for $\eps < \eps_0$, and so it is not clear from the above argument if $U_\eps$ is a minimizer of $E_\eps^{GL}$ in $\mcA^{GL}$. However, in the current case of gradient field configurations (i.e. $V = \nabla v$), we are able to conclude in dimension $N \geq 4$. 

To appreciate the idea, consider the limit $\eps \rightarrow 0$ where $L_\eps \rightarrow -\Delta - \frac{N-1}{r^2} =: L_*$ as bounded linear operators from $H_0^1(B^N,\RR^N)$ into $H^{-1}(B^N,\RR^N)$. Although $L_*$ is not positive definite when $N \leq 6$, we have the following inequality\footnote{Here $H_0^2(B^N,\RR)$ is the closure of $\mathcal{C}_c^\infty(B^N,\RR)$ in $H^2(B^N,\RR)$. In particular, if $v \in H_0^2(B^N,\RR)$, then $v$ and $\nabla v$ have zero trace on the boundary.} for gradient fields in dimension $N \geq 4$:
\[
\int_{B^N} L_* (\nabla v) \cdot (\nabla v)\,dx = \int_{B^N} \Big((\Delta v)^2 - \frac{N-1}{r^2} |\nabla v|^2\Big)\,dx \geq 0 \text{ for } v \in H_0^2(B^N,\RR),
\]
i.e. $L_*$ is positive definite on the subspace of gradient fields in $H^1_0(B^N,\RR^N)$.
This is a consequence of  the sharp Hardy inequality for gradient fields (see e.g. \cite{Beckner08-FM, Cazacu20-PRSE, GhoussoubMoradifam11-MAnn, TertikasZog07-AdvM}):
\begin{equation}\label{eq_HardyRellich}
\int_{\R^N}(\Delta v)^2\,dx\geq c_N\int_{\R^N}\frac{|\nabla v|^2}{r^2}\,dx, \text{ for } v \in H_0^2(B^N,\RR) \text{ where } c_N:=\begin{cases}\frac{N^2}{4}&\text{ if }N\geq 5\\
3 & \text{ if }N=4\\
\frac{25}{36}&\text{ if }N=3,\\
0 & \text{ if } N = 2.\end{cases}
\end{equation}
For the general case $\eps > 0$, we combine the above idea with the machineries in \cite{INSZ18_CRAS, INSZ_AnnENS} and \cite{IN24-IHP}, based on the Hardy decomposition method.

Unfortunately, the above strategy does not work in dimension $N = \{2,3\}$ (for the proof, see  Appendix \ref{app:F<0}):

\begin{proposition}\label{Prop:F<0}
In dimension $N \in \{2,3\}$, there exists a function $v \in \mathcal{C}_c^2(B^N \setminus \{0\}) \subset H_0^2(B^N)$ such that $F_\eps(\nabla v) < 0$ when $\eps$ is sufficiently small.
\end{proposition}

\subsubsection*{Method 2}

As mentioned above, the second method of proof uses a symmetrization procedure. For that, we use the spherical coordinates: for every $x\in B^N$, we write $x=r\theta$ with $r=|x|$ and $\theta\in \SN$. For $v\in H^1(B^N, \R)$, we associate the radial function $\check v=\check v(r)$ given by the  formula
\be
\label{v-rad}
\check v(r) = -\int_r^1 \Big(\fint_{\mathbb{S}^{N-1}} |\nabla v(s\theta)|^2d\sigma(\theta) \Big)^{1/2}\,ds \leq 0, \quad r\in (0,1).
\ee
One can think of this as a kind of rearrangement in the spherical harmonic decomposition of $v$ (see Section \ref{ssec:3.2} for more detailed discussion).

We prove the following.

\begin{theorem}\label{main_rearrangement}
Let $N\geq 2$, $v \in H^1(B^N,\RR)$ and $\check v$ be associated to $v$ by \eqref{v-rad}. The following conclusions hold.
\begin{enumerate}[label=(\roman*)]
\item The map $v \mapsto \check v$ is a Lipschitz continuous map from $H^1(B^N,\RR)$ into $H^1_0(B^N,\RR)$. Moreover, 
\[
\int_{\mathbb{S}^{N-1}} |\nabla \check v(r\theta)|^2\,d\sigma(\theta) = \int_{\mathbb{S}^{N-1}} |\nabla v(r\theta)|^2\,d\sigma(\theta) \text{ for a.e. } r \in (0,1).
\]

\item Let $G: [0,\infty) \times [0,\infty) \rightarrow [0,\infty)$ be continuous. If $G$ is convex in the second variable, then
\begin{align*}
\int_{B^N} G(r,|\nabla \check v|^2)\,dx \leq \int_{B^N} G(r, |\nabla v|^2)\,dx.
\end{align*}
In particular, for any $2 < p < \infty$,
\[
\int_{B^N} |\nabla \check v|^p\,dx \leq \int_{B^N} |\nabla v|^p\,dx.
\]

\item If $v \in H_0^1(B^N)$, i.e. if $v = 0$ on $\partial B^N$ and $1 \leq p \leq 2$, then
\[
 \int_{\SN} |\check v(r\theta)|^p\,d\sigma(\theta) \geq \int_{\SN} |v(r\theta)|^p\,d\sigma(\theta)  \text{ for a.e. } r \in (0,1).
\]

\item Assume in addition that $v\in H^2(B^N,\R)$ with the boundary condition $\nabla v(x)=cx$ on $\partial B^N$ for some constant $c \in \RR$. Then $\check v \in H^2(B^N,\R)$ and  $\nabla \check v(x) = |c|x$ on $\partial B^N$. If $N \geq 5$, then
\begin{equation}\label{Eq:RearrIneql}
\int_{B^N} (\Delta\check v)^2\,dx \leq \int_{B^N} (\Delta v)^2\,dx.
\end{equation}

If $N\in\{2, 3,4\}$, \eqref{Eq:RearrIneql} continues to hold provided that $\int_{\mathbb{S}^{N-1}}v(r\theta) \theta\,d\sigma(\theta)=0$ for a.e. $r \in (0,1)$. In either case, equality is attained if and only if $v$ is radially symmetric and $|v'| = \check v'$ in $(0,1)$.
\end{enumerate}
\end{theorem}

To apply Theorem \ref{main_rearrangement} to prove Theorem \ref{thm:AG}, we only need to note that for $\nabla u \in \mcA^{GL}$, by integrating by parts using $\nabla u(x) = x$ on $\partial B^N$,
\begin{align*}
\int_{B^N} |\nabla^2 u|^2\,dx 
	&= \int_{B^N} (\Delta u)^2\,dx - \int_{\mathbb{S}^{N-1}} \underbrace{\sum_{i,j=1}^N \partial_{i} \partial_j u(\theta) (\delta_{ij} - \theta_i \theta_j)}_{\nabla^2 u : (I_N - \theta \otimes \theta)}\,d\sigma(\theta)\\
	&= \int_{B^N} (\Delta u)^2\,dx - \int_{\mathbb{S}^{N-1}} \big[(N-1) \partial_r u(\theta) + \Delta_{\SN} u(\theta)\big]\,d\sigma(\theta)\\
	&= \int_{B^N} (\Delta u)^2\,dx -(N-1)|\mathbb{S}^{N-1}|.
\end{align*}
(Here we have used the fact that $I_N - \theta \otimes \theta$ is the projection onto the tangent hyperplane $T_\theta \SN$.) 
Therefore, in dimension $N \geq 5$, Theorem \ref{main_rearrangement} immediately implies that minimizers of $E_\eps^{GL}$ in $\{U=\nabla u \in \mcA^{GL}\}$ are radially symmetric. Thanks to the characterization of radially symmetric critical points in \cite{IN24-IHP}, Theorem \ref{thm:AG} follows. Theorem \ref{thm2} also follows in a similar manner. For Theorem \ref{thm1}, we need an extra symmetrization for the $U_{N+1}$ component; see Section \ref{ssec:4.1}.

In Theorem \ref{main_rearrangement}(iv), the requirement $\int_{\mathbb{S}^{N-1}}v(r\theta) \theta\,d\sigma(\theta)=0$ in dimension $N \in \{2, 3, 4\}$ cannot simply be dropped due to existence of counter-examples. (For examples of symmetry breaking phenomena in the context of Hardy's inequality for gradient fields in dimension $N \in \{3,4\}$, see e.g. \cite{Cazacu20-PRSE}.)

Our rearrangement is related to a vectorial rearrangement in Lieb and Loss \cite{LiebLoss95-JEDP}. For $V\in H^1(B^N,\R^N)$, one associates the radially symmetric vector field $\check V$ defined by 
\begin{equation}\label{def_sym}
\check{V}(x)= \left(\fint_{\mathbb{S}^{N-1}}|V(r\theta)|^2d\sigma(\theta)\right)^\frac{1}{2}\frac{x}{r}.
\end{equation}
It was shown in \cite{LiebLoss95-JEDP} that, provided $V(x) = x$ on $\partial B^N$ and $\int_{\mathbb{S}^{N-1}} V(r\theta)d\sigma(\theta) = 0$ for a.e. $r \in (0,1)$,
\[
\int_{B^N} |\nabla \check V|^2\,dx \leq \int_{B^N} |\nabla V|^2\,dx.
\]
It is readily seen that if $V = \nabla v$, then $\check V = \nabla \check v$. Thus, when $N \in \{2, 3,4\}$, the conclusion \eqref{Eq:RearrIneql} in Theorem \ref{main_rearrangement} can be deduced from the above result in \cite{LiebLoss95-JEDP}.

Fewer rearrangement methods are known to prove symmetry of solutions of higher order elliptic equations than for second order ones. This can be partly explained by the absence of a maximum principle for higher order elliptic equations or systems, which makes Schwarz symmetrization methods inapplicable in general. There are some exceptions, see for instance the two papers of Nadirashvili \cite{Nadirashvili95} and Talenti \cite{Talenti81} where it is shown by rearrangement arguments that minimizers of $\frac{|\{u\neq 0\}|^2\int_{\R^2} (\Delta u)^2}{\int_{\R^2}u^2}$ are radially symmetric.\footnote{For other results on symmetry of solutions of higher order elliptic equations or systems which do not use rearrangement inequalities, see e.g. \cite{Berchio08, Sirakov05, ferrero07, Gazzola10} and the references therein.}

More recently, a rearrangement principle developed in Lenzmann and Sok \cite{lenzmann21} deals with the radial symmetry of optimizers of Gagliardo--Nirenberg type inequalities of arbitrarily high orders, as well as  ground states of higher order non-linear Schrödinger equations of the form
\[Lv+\omega v=v|v|^{p-2}\text{ in }\R^N\]
where $L$ is a certain pseudodifferential operator, $\omega>0$ and $p\in (2,p^*)$ for some critical exponent $p^* > 2$ depending only on the dimension $N$ and the operator $L$. The rearrangement principle here is based on Schwarz rearrangement of the Fourier transform: any function $v:\R^N\to \mathbb{C}$ is symmetrized as $v^\sharp=\mathcal{F}^{-1}\left[|\mathcal{F}[v]|^*\right]$, where $\mathcal{F}$ is the Fourier transform and $w^*$ designates the radially decreasing Schwarz rearrangement of $w$.

We make a comparison between the rearrangement $\check v$ in Theorem \ref{main_rearrangement} and the rearrangement $v^\sharp$ of \cite{lenzmann21} in the following table.

\begin{center}
\begin{tabular}{c|c|c}
&$\check v$ on $B^N$ & $v^\sharp$ on $\RR^N$ \\
\hline
$L^2$-norm & $\|\check v\|_{L^2} \geq \|v\|_{L^2}$ & $\|v^\sharp\|_{L^2} = \|v\|_{L^2}$\\
\hline
$L^p$-norm, $1 \leq p < 2$ & $\|\check v\|_{L^p} \geq \|v\|_{L^p}$ & ?\\
\hline
$L^p$-norm, even integer $p > 2$ & ? & $\|v^\sharp\|_{L^p} \geq \|v\|_{L^p}$\\
\hline
\parbox{1in}{\centering $\dot H^1$-norm} & $\|\nabla\check v\|_{L^2} = \|\nabla v\|_{L^2}$ & $\|\nabla v^\sharp\|_{L^2} \leq \|\nabla v\|_{L^2}$\\
\hline
\parbox{1.5in}{\centering $\dot W^{1,p}$-norm, $p > 2$} & $\|\nabla\check v\|_{L^p} \leq \|\nabla v\|_{L^p}$ & ?\\
\hline
\parbox{1in}{\centering $\dot H^2$-norm} & $\|\Delta\check v\|_{L^2} \leq \|\Delta v\|_{L^2}$ & $\|\Delta v^\sharp\|_{L^2} \leq \|\Delta v\|_{L^2}$\\
\hline
\parbox{2in}{\centering $\dot H^s$-norm, $s > 0$} & ? & $\| v^\sharp\|_{\dot H^s} \leq \|v\|_{\dot H^s}$
\end{tabular}
\end{center}

As an application Theorem \ref{main_rearrangement}, we consider the radial symmetry of `ground state' solutions to the nonlinear eigenvalue problem 
\begin{equation}
\begin{cases}
\Delta^2 v  = \lambda v + |v|^{p-2} v \text{ on } B^N,\\
v = \partial_r v = 0 \text{ on } \partial B^N,
\end{cases}
	\label{Eq:BHE}
\end{equation}
where $1 \leq p < 2$. See Section \ref{ssec:RApp}.

\bigskip
\noindent{\bf Acknowledgments.} R.I. is partially supported by the ANR projects ANR-21-CE40-0004
and ANR-22-CE40-0006-01. L.N. was partially supported by the ANR LabEx CIMI (grant ANR-11-LABX-0040) within the French State Programme ``Investissements d'Avenir'' during his visit to Universit\'e Toulouse III -- Paul Sabatier. Part of this work was done while all authors visited the Hausdorff Research Institute for Mathematics in Bonn, funded by the Deutsche Forschungsgemeinschaft (DFG, German Research Foundation) under Germany's Excellence Strategy -- EXC-2047/1 -- 390685813, as part of the Trimester Program on ``Mathematics of Complex Materials''.

\section{First proof of main results}
\label{sec2}

\subsection{Proof of Theorem \ref{thm1}}

In this section we give the first proof of Theorem \ref{thm1} based on the strategy in \cite{INSZ18_CRAS, INSZ_AnnENS}  and exploiting the additional structure of a gradient field for the first $N$-components of the current admissible configurations.

Recall that $\ell(\eps)$ is the first eigenvalue of the operator $L_\eps = -\Delta - \frac{1}{\eps^2} W'(1-f_\eps^2)$ and that the escaping radially symmetric critical points $U_{\eps,\eta}^\pm$ with $g_{\eps,\eta} > 0$ exist if and only if $W'(1) > 0$ and $\ell(\eps)+\frac{1}{\eta^2}\tilde{W}'(0)<0$ (equivalently $0 < \eps < \eps_0$ and $\eta > \eta_0(\eps)$). For fixed $\eps > 0, \eta > 0$, we let 
\[
\Phi  
	=\begin{cases}U_{\eps,\eta}^+&\text{ if } W'(1) > 0 \text{ and } \ell(\eps)+\frac{1}{\eta^2}\tilde{W}'(0)<0 \text{ (i.e. there is an escaping solution),}\\
\bar U_\eps&\text{ otherwise}\text{ (i.e. there is no escaping solution),}\end{cases}
\]
and
\[
(f,g)
	=\begin{cases}(f_{\eps,\eta},g_{\eps,\eta})&\text{ if } W'(1) > 0 \text{ and } \ell(\eps)+\frac{1}{\eta^2}\tilde{W}'(0)<0,\\
(f_\eps,0)&\text{ otherwise},
	\end{cases}
\]
so that $\Phi(x) = (f(r) \frac{x}{r}, g(r))$.

We consider  the differential operators $L_{\eps,\eta}$ and $T_{\eps,\eta}$:
\[
L_{\eps,\eta}=-\Delta - \frac{1}{\eps^2} W'(1-f^2-g^2), \quad T_{\eps,\eta}=-\Delta - \frac{1}{\eps^2} W'(1-f^2-g^2) + \frac{1}{\eta^2} \tilde W'(g^2).
\]
For any $v\in H^2_0(B^N,\R)$, we let
\begin{equation}
F_{\eps,\eta}[\nabla v]=\int_{B^N}\left((\Delta v)^2- \frac{1}{\eps^2}  W'(1-f^2-g^2)|\nabla v|^2\right)\,dx = \int_{B^N} L_{\eps,\eta}(\nabla v) \colon (\nabla v)\,dx.
	\label{Eq:Feedef}
\end{equation}
Note that  $\int_{B^N}|\nabla^2 v|^2\,dx =\int_{B^N}(\Delta v)^2\,dx$ since $v\in H^2_0(B^N)$. Note also that, in the non-escaping case, $\Phi = \bar U_\eps$, $L_{\eps,\eta} = L_\eps$, $T_{\eps,\eta} = \bar T_{\eps,\eta}$ and $F_{\eps,\eta} = F_\eps$ introduced in sections \ref{ssec:1.3} and \ref{ssec:1.4}. 

As in \cite{INSZ18_CRAS, INSZ_AnnENS}, the starting point of the proof is the following consequence of the convexity of $W$ and $\tilde W$:
\begin{lemma}
\label{lem_conservationgradu}
For any $v\in H^2_0(B^N,\R)$ and $p\in H^1_0(B^N,\R)$,
\[E_{\eps,\eta}\big[\Phi +(\nabla v,p)\big]-E_{\eps,\eta}[\Phi]
	\geq\frac{1}{2}F_{\eps,\eta}[\nabla v]+\frac12 \int_{B^N} T_{\eps,\eta}p\cdot p\, dx.
\]
\end{lemma}
\begin{proof} We have
\begin{align}
E_{\eps,\eta}\big[\Phi+(\nabla v,p)\big]-E_{\eps,\eta}[\Phi]
	&= \frac{1}{2}\int_{B^N}
		\Big\{ 2 \nabla \Phi : \nabla(\nabla v, p) + |\nabla^2 v|^2 + |\nabla p|^2\nonumber\\
			&\qquad 
			+ \frac{1}{\eps^2} [W(1-|\Phi+(\nabla v,p)|^2)-W(1-|\Phi|^2)] \nonumber\\
			&\qquad + \frac{1}{\eta^2}[\tilde{W}((g+p)^2)-\tilde{W}(g^2)]
		\Big\}\,dx.\label{Eq:EEexp}
\end{align}
By the convexity of $W$ and $\tilde W$, we have
\begin{align*}
W(1-|\Phi+(\nabla v,p)|^2)-W(1-|\Phi|^2)
	&\geq W'(1-|\Phi|^2)\left(|\Phi|^2-|\Phi+(\nabla v,p)|^2\right)\\
	&= -W'(1-|\Phi|^2)\left(2\Phi \cdot(\nabla v,p) + |\nabla v|^2 + p^2\right)\\
\tilde{W}((g+p)^2)-\tilde{W}(g^2)
	&\geq \tilde{W}'(g^2)\left((g+p)^2 - g^2\right)\\
	&= \tilde{W}'(g^2)\left(2g p + p^2\right).
\end{align*}
Since $\Phi$ is a critical point of $E$, we also have
\[
\int_{B^N}
		\Big\{  \nabla \Phi : \nabla(\nabla v, p) 
			- \frac{1}{\eps^2} W'(1-|\Phi|^2)\Phi \cdot(\nabla v,p) \\
			+ \frac{1}{\eta^2} \tilde{W}'(g^2) g p
		\Big\}\,dx = 0.
\]
Inserting the last two estimates into \eqref{Eq:EEexp} we arrive at 
\begin{align*}
&E_{\eps,\eta}\left[\Phi+(\nabla v,p)\right]-E_{\eps,\eta}[\Phi]\\
	&\qquad \geq \frac{1}{2}\int_{B^N}\left(|\nabla^2 v|^2+|\nabla p|^2- \frac{1}{\eps^2}W'(1-f^2-g^2)\left(|\nabla v|^2+p^2\right)+ \frac{1}{\eta^2} \tilde{W}'(g^2)p^2\right)\,dx,
\end{align*}
which is precisely the conclusion. 
\end{proof}

We will frequently make use of the following Hardy decomposition:

\begin{lemma}[{\cite[Lemma A.1]{INSZ3}}]\label{Lem:HardyD}
Let $A: B^N \rightarrow \R^{N \times N}$ be a $C^1$ non-negative semi-definite symmetric form, i.e. $A(x) \xi \cdot \xi \geq 0$ for every $x \in B^N$ and $\xi \in \RR^N$. We define the operator
\[
L := -\nabla \cdot(A \nabla ) 
\]
and consider a smooth positive function $\psi: B^N \rightarrow \RR$. Then for every $u \in \mathcal{C}_c^\infty(B^N,\R)$, we have the following Hardy decomposition:
\[
\int_{B^N} Lu \cdot u\,dx = \int_{B^N} \psi^2 A(x) \nabla (\frac{u}{\psi}) \cdot \nabla (\frac{u}{\psi})\,dx + \int_{B^N} \frac{u^2}{\psi^2} L\psi \cdot \psi\,dx.
\]
\end{lemma}

Before moving on with the proof, let us make a simple observation on the non-negativity of $T_{\eps,\eta}$.
\begin{lemma}\label{Lem:l1T}
The first eigenvalue of $T_{\eps,\eta}$ on $H^1_0(B^N,\R)$ is $\left(\ell(\eps)+\frac{1}{\eta^2}\tilde{W}'(0) \right)_+$ and the corresponding first eigenspace of $T_{\eps,\eta}$
\begin{itemize}
\item coincides with the first eigenspace of $L_\eps$ when $\ell(\eps)+\frac{1}{\eta^2}\tilde{W}'(0) \geq 0$ (i.e. when $g \equiv 0$), and
\item is generated by $g$ when $\ell(\eps)+\frac{1}{\eta^2}\tilde{W}'(0) < 0$ (i.e. when $g > 0$).
\end{itemize}
In particular $T_{\eps,\eta}$ is non-negative semi-definite on $H^1_0(B^N,\R)$ and
\[
\int_{B^N} T_{\eps,\eta}p\cdot p\, dx\geq \int_{B^N} \Big[h^2 \left|\nabla(\frac{p}{h})\right|^2+\left(\ell(\eps)+\frac{1}{\eta^2}\tilde{W}'(0)\right)_+p ^2\Big]\,dx \geq 0,
\]
where $h$ is any first eigenfunction of $T_{\eps,\eta}$.
\end{lemma}

\begin{proof}
Recall that, by \cite[Theorem 2.4]{IN24-IHP} on escaping and non-escaping critical points of $E_{\eps,\eta}$, when $g\equiv 0$, we have $\ell(\eps)+\frac{1}{\eta^2}\tilde{W}'(0)  \geq 0$, while, when $g > 0$, $\ell(\eps)+\frac{1}{\eta^2}\tilde{W}'(0) < 0$. The first bullet point is then clear as  $T_{\eps,\eta} = L_\eps + \frac{1}{\eta^2}\tilde{W}'(0)$ and the first eigenvalue of $L_\eps$ is $\ell(\eps)$. When $g > 0$, we have
\[
T_{\eps,\eta} g = 0,
\]
and so $g$ must be a first eigenfunction of $T_{\eps,\eta}$ and the first eigenvalue of $T_{\eps,\eta}$ must be zero. The second bullet point follows. The last assertion follows from the Hardy decomposition Lemma \ref{Lem:HardyD} with the decomposition $p = h \frac{p}{h}$.
\end{proof}

The last ingredient for the proof of Theorem \ref{thm1} is:
\begin{proposition}
\label{pro}
Suppose $N \geq 4$. For any $v\in H^2_0(B^N,\R)$ we have
\[F_{\eps,\eta}[\nabla v]\geq \frac{(N-2)^2}{4}\int_{B^N}\frac{(\partial_r v)^2}{r^2}\,dx +\left(\frac{N^2}{2}-2N\right)\int_{B^N}\frac{|\nabla v|^2-(\partial_r v)^2}{r^2}\,dx \geq 0.
\]
\end{proposition}

\begin{remark}
Note that for general $V \in H_0^1(B^N,\R^N)$ which is not necessarily a gradient field, it was shown in \cite{INSZ18_CRAS, INSZ_AnnENS}  in dimension $N \geq 7$ that $F_{\eps,\eta}[V] = F_\eps[V] \geq 0$. 
\end{remark}

Before giving the proof of the above proposition, let us prove Theorem \ref{thm1}.

\begin{proof}[First proof of Theorem \ref{thm1}]
Indeed, as $N\geq 4$, we have by Proposition \ref{pro} that $F_{\eps,\eta}[\nabla v]\geq 0$ for every $v\in H^2_0(B^N,\R)$ with equality if and only if $\partial_r v = 0$ a.e., which implies $v=0$. Therefore, by Lemmas \ref{lem_conservationgradu} and \ref{Lem:l1T}, $\Phi$ is a minimizer of our problem. If $\tilde \Phi$ is another minimizer of $E_{\eps,\eta}$, then $E_{\eps,\eta}[\tilde \Phi] = E_{\eps,\eta}[\Phi]$. By Lemmas \ref{lem_conservationgradu}, \ref{Lem:l1T} and Proposition \ref{pro}, this is possible only if $\tilde \Phi - \Phi = (0,h)$ for some $h$ in the first eigenspace of $T_{\eps,\eta}$, which is radially symmetric (because $T_{\eps,\eta}$ is radially symmetric). We thus have that $\tilde \Phi$ is a radially symmetric minimizer of $E_{\eps,\eta}$. \cite[Therorem 2.4]{IN24-IHP} then gives the desired uniqueness for minimizer(s).
\end{proof}

\begin{proof}[Proof of Proposition \ref{pro}]
It is enough to prove the estimate for $v\in\mathcal{C}^{\infty}_c(B^N\setminus\{0\},\R)$. The general case follows from Fatou's lemma and the density of $\mathcal{C}^{\infty}_c(B^N\setminus\{0\},\R)$ in $H^2_0(B^N, \R)$ (note $N \geq 4$).

We denote by $(\phi_k)_{k\in\N}$ an orthonormal basis of $L^2(\SN)$ given by eigenfunctions of the Laplace-Beltrami operator on the unit sphere, meaning that for any $k\in\N$ we have
\[-\Delta_{\SN}\phi_k=\lambda_k\phi_k\]
where $0=\lambda_0<N-1=\lambda_1=\hdots =\lambda_N<2N=\lambda_{N+1} \leq \cdots \longrightarrow +\infty$. In particular we have
\begin{equation}
\int_{\SN}\phi_k\phi_l \,d\sigma(\theta)=\delta_{kl} \quad \text{ and } \quad \int_{\SN}\nabla_{\SN}\phi_k\cdot\nabla_{\SN}\phi_l\, d\sigma(\theta)=\lambda_k\delta_{kl}.
	\label{Eq:OrRel}
\end{equation}
Consider the decomposition of $v$ in spherical harmonics: we write
\[
v(r\theta)=\sum_{k\geq 0}v_k(r)\phi_k(\theta)  \text{ for } r\in (0,1), \theta\in\SN
\]
where $v_k\in\mathcal{C}^\infty_c((0,1),\R)$. We have 
\[\nabla v=\sum_{k\geq 0}\left(v_k'\phi_k\frac{x}{r}+ \frac{1}{r}v_k\nabla_{\mathbb{S}^{N-1}}\phi_k\right),\ \Delta v=\sum_{k\geq 0}\left(v_k''+ \frac{N-1}{r} v_k'- \frac{\lambda_k}{r^{2}}v_k\right)\phi_k.
\]
Using the orthogonality relations \eqref{Eq:OrRel} and the identities
\begin{align*}
\int_0^1 r^{N-2} v_k'' v_k'\,dr &=  -\frac{N-2}{2} \int_0^1 r^{N-3} (v_k')^2\,dr,
\\
\int_0^1 r^{N-4} v_k' v_k\,dr &=  -\frac{N-4}{2} \int_0^1 r^{N-5} v_k^2\,dr
	\text{ for } k \geq 1, \\
\int_0^1 r^{N-3} v_k'' v_k\,dr &=  \int_0^1 \Big[-r^{N-3} (v_k')^2 + \frac{(N-3)(N-4)}{2}  r^{N-5} v_k^2\Big]\,dr\text{ for } k \geq 1,
\end{align*}
we get
\begin{equation}\label{eq_intDeltav}
\begin{split}
&\int_{B^N} (\Delta v)^2dx =\sum_{k\geq 0}\int_{B^N}\left(v_k''+ \frac{N-1}{r} v_k'- \frac{\lambda_k}{r^{2}}v_k\right)^2\phi_k^2dx\\
&\quad =\sum_{k\geq 0}\int_{0}^{1}\left(r^{N-1}(v_k'')^2+(N-1+2\lambda_k)r^{N-3}(v_k')^2+\lambda_k(\lambda_k+2N-8))r^{N-5}v_k^2\right)dr,
\end{split}
\end{equation}
and
\begin{align*}
\int_{B^N}W'(1-f^2-g^2)|\nabla v|^2dx&=\sum_{k\geq 0}\int_{0}^{1}W'(1-f^2-g^2)\left(r^{N-1}(v_k')^2+\lambda_k r^{N-3}v_k^2\right)dr.
\end{align*}

Inserting these into \eqref{Eq:Feedef}, we split $F_{\eps,\eta}$ into three terms as follows:
\begin{align*}
F_{\eps,\eta}[\nabla v]=&\sum_{k\geq 0}\Big\{\underbrace{\int_{0}^{1}r^{N-1}\left((v_k'')^2- \frac{1}{\eps^2}W'(1-f^2-g^2)(v_k')^2\right)\,dr}_{\text{I}_k}\\
	&\qquad+\underbrace{\int_0^1 \lambda_k r^{N-1}\left(r^{-2}(v_k')^2- \frac{1}{\eps^2}W'(1-f^2-g^2)r^{-2}v_k^2\right)\,dr}_{\text{II}_k}\\
&\qquad +\underbrace{\int_{0}^{1}\left((N-1+\lambda_k)r^{N-3}(v_k')^2+\lambda_k(\lambda_k+2N-8)r^{N-5}v_k^2\right)\,dr}_{\text{III}_k}\Big\}.
\end{align*}

For terms $\text{I}_k$ and $\text{II}_k$ we will apply the Hardy decomposition Lemma \ref{Lem:HardyD} using  
$$L_{\eps,\eta} f=-\frac{N-1}{r^2}f.$$ 
More precisely, for any function $w\in\mathcal{C}^{\infty}_c(B^N, \R)$ we have the identity
\begin{align}
\int_{B^N} L_{\eps,\eta}(fw)\cdot (fw)\, dx
	&=\int_{B^N}\big(f^2 |\nabla w|^2+w^2 L_{\eps,\eta}f\cdot f\big)\,dx\nonumber\\
	&=\int_{B^N}f^2\left(|\nabla w|^2-\frac{N-1}{r^2}w^2\right)\,dx.
	\label{eq_identityf}
\end{align}

\begin{itemize}
\item \textbf{Estimate of }$\text{I}_k$: For the first term we use the decomposition $v_k' = f \frac{v_k'}{f}$, i.e. $w = \frac{v_k'}{f} \in\mathcal{C}^{\infty}_c(B^N \setminus \{0\}, \R)$ in \eqref{eq_identityf}:
\[
\text{I}_k
	=\int_{0}^{1}r^{N-1}L_{\eps,\eta}(v_k') \cdot(v_k')dr
	=\int_{0}^{1}\Big[r^{N-1}f^2\left|\left(\frac{v_k'}{f}\right)'\right|^2-(N-1)r^{N-3}(v_k')^2\Big]dr.
\]

We let $\zeta(r)=r^{-\frac{N-2}{2}}$ so that, when seen as a radial function in $\R^N \setminus \{0\}$, $\zeta$ verifies 
\[
-\nabla\cdot (f^2 \nabla \zeta) = -f^2 \Delta \zeta - 2f f' \zeta' = \frac{(N-2)^2}{4r^2} f^2 \zeta - 2 f f' \zeta' \geq \frac{(N-2)^2}{4r^2} f^2 \zeta,
\]
since $\zeta' < 0$  and $f, f' > 0$ in $(0,1)$. By the Hardy decomposition Lemma \ref{Lem:HardyD}  for the operator $\nabla \cdot (f^2 \nabla)$ and the decomposition $\frac{v_k'}{f} = \zeta \frac{v_k'}{f \zeta }$, we thus have
\begin{equation}\label{eq_estI_k}
\begin{split}
\text{I}_k&\geq\int_{0}^{1}r^{N-1}\left(f^2\zeta^2\left|\left(\frac{v_k'}{f \zeta }\right)'\right|^2+\left(\frac{(N-2)^2}{4}-(N-1)\right)r^{N-3}(v_k')^2\right)dr.
\end{split}
\end{equation}

\item \textbf{Estimate of }$\text{II}_k$: First notice the elementary identity
\begin{align*}
\int_{0}^{1}r^{N-3}(v_k')^2 dr
	&=\int_{0}^{1}\left(r^{N-1}\Big(\left(\frac{v_k}{r}\right)'\Big)^2+2r^{N-4}v_kv_k'-r^{N-5}v_k^2\right)dr\\
	&=\int_0^1\left(r^{N-1}\Big(\left(\frac{v_k}{r}\right)'\Big)^2-(N-3)r^{N-5}v_k^2\right)dr
\end{align*}
so
\[
\text{II}_k=\lambda_k\int_{0}^{1}\left(r^{N-1}L_{\eps,\eta}(\frac{v_k}{r}) \cdot  (\frac{v_k}{r}) -(N-3)r^{N-5}v_k^2\right)dr.
\]
This time we use the decomposition $\frac{v_k}{r} = f\frac{v_k}{rf}$ (i.e. $w = \frac{v_k}{rf}$ in \eqref{eq_identityf}) to obtain
\[
\text{II}_k=\lambda_k\int_{0}^{1}\left(r^{N-1}f^2\left|\left(\frac{v_k}{rf}\right)'\right|^2-2(N-2)r^{N-5}v_k^2\right)dr.
\]

By the Hardy decomposition Lemma \ref{Lem:HardyD} for the operator $\nabla \cdot (f^2 \nabla)$ and the decomposition $\frac{v_k}{rf} = \zeta \frac{v_k}{r f \zeta}$ as above we get the estimate
\begin{equation}\label{eq_estII_k}
\begin{split}
\text{II}_k&\geq \lambda_k\int_{0}^{1}\left(r^{N-1}f^2\zeta^2\left|\left(\frac{v_k}{rf\zeta}\right)'\right|^2+\left(\frac{(N-2)^2}{4}-2(N-2)\right)r^{N-5}v_k^2\right)dr.
\end{split}
\end{equation}
\item \textbf{Estimate of }$\text{III}_k$: For the last term we simply apply the Hardy inequality once: for any $v\in\mathcal{C}^\infty_c((0,1),\R)$, $\int_{0}^{1}r^{N-3}(v')^2 dr \geq \frac{(N-4)^2}{4}\int_{0}^{1}r^{N-5}v^2dr$. This gives
\begin{equation}\label{eq_estIII_k}
\text{III}_k\geq \int_{0}^{1}\left((N-1)r^{N-3}(v_k')^2+\lambda_k\left(\lambda_k+2N-8+\frac{(N-4)^2}{4}\right)r^{N-5}v_k^2\right)dr.
\end{equation}
\end{itemize}
Summing the estimates \eqref{eq_estI_k}, \eqref{eq_estII_k}, \eqref{eq_estIII_k} we get
\begin{align*}
F_{\eps,\eta}[\nabla v]&\geq \sum_{k\geq 0}\int_{0}^{1}\left(\frac{(N-2)^2}{4}r^{N-3}(v_k')^2+\lambda_k\left(\frac{N^2}{2}-3N+1+\lambda_k\right)r^{N-5}v_k^2\right)dr\\
&\geq \sum_{k\geq 0}\int_{0}^{1}\left(\frac{(N-2)^2}{4}r^{N-3}(v_k')^2+\lambda_k\left(\frac{N^2}{2}-2N\right)r^{N-5}v_k^2\right)dr\text{ since }\lambda_k^2\geq (N-1)\lambda_k\\
&=\frac{(N-2)^2}{4}\int_{B^N}\frac{(\partial_r v)^2}{r^2}dx +\left(\frac{N^2}{2}-2N\right)\int_{B^N}\frac{|\nabla v|^2-(\partial_r v)^2}{r^2}dx.
\end{align*}
The result is proved.
\end{proof}

\subsection{Proof of Theorems \ref{thm:AG} and \ref{thm2}}

Theorem \ref{thm:AG} for the Aviles--Giga model is a simple consequence of Theorem \ref{thm1} for the extended model.

\begin{proof}[Proof of Theorem \ref{thm:AG}]
Fix $\eps > 0$. Pick any convex $C^2$ function $\tilde W: [0,\infty) \rightarrow [0,\infty)$ with $\tilde W(0) = 0$ and $\tilde W'(0) > 0$, e.g. $\tilde W(t) = t$. By \cite{IN24-IHP}, there exists a small $\eta > 0$ such that $E_{\eps,\eta}$ has no escaping radially symmetric critical points. By Theorem \ref{thm1}, $\bar U_\eps = (U_\eps,0)$ is the unique minimizer of $E_{\eps,\eta}$ in $\mcA$. It follows that
\[
E_\eps^{GL}[\nabla u] = E_{\eps,\eta}[(\nabla u,0)] \geq E_{\eps,\eta}[\bar U_\eps] = E_\eps[U_\eps] \text{ for all }  \nabla u \in \mcA^{GL}.
\]
This means that $U_\eps$ is a minimizer of $E_\eps^{GL}$ in $\{\nabla u \in \mcA^{GL}\}$. Conversely, if $\nabla \tilde u$ is a minimizer of $E_\eps^{GL}$ in $\{\nabla u \in \mcA^{GL}\}$, then
\[
E_{\eps,\eta}[(\nabla \tilde u,0)] = E_\eps^{GL}[\nabla \tilde u] = E_\eps^{GL}[U_\eps]  = E_{\eps,\eta}[\bar U_\eps],
\]
i.e. $(\nabla \tilde u,0)$ is also a minimizer of $E_{\eps,\eta}$ in $\mcA$. By Theorem \ref{thm1}, $\nabla \tilde u = U_\eps$ as desired.
\end{proof}

We next prove Theorem \ref{thm2} for the $\mathbb{S}^N$-valued Ginzburg--Landau model.

\begin{proof}[Proof of Theorem \ref{thm2}]
Set $W(t)=t^2$ and fix some $\eta>0$. As $4 \leq N \leq 6$ and $W'(1) > 0$, we know by \cite{IN24-IHP} that for $\eps>0$ small enough, there exists a unique escaping radially symmetric critical point of the form
$$
U_{\eps, \eta}=(f_{\eps,\eta}(r)\frac{x}{r},g_{\eps,\eta}(r)) \in \mcA, \quad g_{\eps,\eta}>0 \textrm{ in } (0,1)
$$
of the energy $E_{\eps, \eta}$. Pick an arbitrary $M=(\nabla m, M_{N+1})\in \mcA^{MM}$ (in particular, $|M|=1$) and set
$$(\nabla v_{\eps, \eta}, p_{\eps, \eta}):=M-U_{\eps, \eta}.$$
Then by Section \ref{sec2}, we know that
\begin{align*}
E^{MM}_\eta [M]&=E_{\eps, \eta}\big[U_{\eps, \eta}+(\nabla v_{\eps, \eta},p_{\eps, \eta})\big]\\
&\geq E_{\eps, \eta}[U_{\eps, \eta}]+\frac{1}{2}F_{\eps, \eta}[\nabla v_{\eps, \eta}]+\frac12 \int_{B^N} T_{\eps, \eta}p_{\eps, \eta}\cdot p_{\eps, \eta}\, dx
\end{align*}
with
\begin{align*}
F_{\eps, \eta}[\nabla v_{\eps, \eta}]
	&\geq \frac{(N-2)^2}{4}\int_{B^N}\frac{(\partial_r v_{\eps, \eta})^2}{r^2}\,dx +\left(\frac{N^2}{2}-2N\right)\int_{B^N}\frac{|\nabla_{\SN} v_{\eps, \eta}|^2}{r^4}\,dx,\\
\int_{B^N} T_{\eps, \eta}p_{\eps, \eta}\cdot p_{\eps, \eta}\, dx
	&\geq \int_{B^N} g^2_{\eps,\eta} \Big|\nabla \big(\frac{p_{\eps,\eta}}{g_{\eps,\eta}}\big)\Big|^2\,dx.
\end{align*}
By \cite[Remark 2.17]{IN24-IHP}, for a subsequence $\eps\to 0$, we have that $U_{\eps,\eta}\to M_{\eta}^+$ in $H^1(B^N)$ (in particular,
$\nabla U_{\eps,\eta}\to \nabla M_{\eta}^+$ and $U_{\eps,\eta}\to M_{\eta}^+$ a.e. in $B^N$) and $E_{\eps, \eta}(U_{\eps, \eta})\to E^{MM}_\eta [M_\eta^+]$ where $M_\eta^+ = (\tilde f_\eta \frac{x}{r}, g_\eta)$ is the unique escaping radially symmetric critical point of $E^{MM}_\eta$ with $g_\eta > 0$ in $(0,1)$. Therefore, $$(\nabla v_{\eps, \eta}, p_{\eps, \eta})\to M-M_\eta^+ =:(\nabla \tilde v_{\eta}, \tilde p_{\eta})$$ in $H^1(B^N)$ and a.e. in $B^N$ as well as $\nabla(\nabla v_{\eps, \eta}, p_{\eps, \eta})\to \nabla (\nabla \tilde v_{\eta}, \tilde p_{\eta})$ a.e. in $B^N$  for a subsequence $\eps\to 0$. By Fatou's lemma, it follows for a subsequence $\eps\to 0$:
\begin{align*}
E^{MM}_\eta [M]
	&=E^{MM}_\eta [M_\eta^++(\nabla \tilde v_{\eta}, \tilde p_{\eta})]\\
	&\geq E^{MM}_{\eta}[M_{\eta}^+]+\frac12 \int_{B^N} g^2_{\eta} \Big|\nabla\big( \frac{\tilde p_{\eta}}{g_{\eta}}\big)\Big|^2\, dx\\
		&\quad + \frac{(N-2)^2}{8}\int_{B^N}\frac{(\partial_r \tilde v_{\eta})^2}{r^2}+\frac12\left(\frac{N^2}{2}-2N\right)\int_{B^N}\frac{|\nabla_{\SN} \tilde v_{\eta}|^2}{r^4}.
\end{align*}
We conclude to the minimality of $M_\eta^+$. If $M$ is another minimizer, within the above notations, then $E_\eta^{MM}[M] = E_\eta^{MM}[M_\eta^+]$ and so $\partial_r \tilde v_{\eta}=0$ in $B^N$ yielding $\tilde v_{\eta}=0$ (as $\tilde v_{\eta}=0$ on $\partial B^N$);  also, $\tilde p_{\eta}=\alpha g_\eta$ for some constant $\alpha\in \R$. Since $|M|=1$ and $M=(0,\tilde p_\eta)+M_\eta$, we deduce that $(\tilde p_{\eta}+g_\eta)^2=g_\eta^2$ yielding $\alpha=0$ or $-2$, i.e. $M = M_\eta^+$ or $M = M_\eta^{-}$.
\end{proof}

\section{Symmetrization and second proof of main results in dimension $N \geq 5$}

\subsection{A symmetrization of scalar functions}\label{ssec:4.1}

In this section, we consider a spherical average rearrangement which is probably known to the experts. See e.g. \cite[Chapter 1, Section 9]{Treves75-Book} for a similar rearrangement in the context of the Laplace operator. Let $1 \leq q < \infty$. For a function $g\in L^q(B^N, \R)$, define a radial symmetrization $\check g$ of $g$ by
\begin{equation}
\check g(r) = \Big\{\fint_{\mathbb{S}^{N-1}} |g(r\theta)|^q\,d\sigma(\theta) \Big\}^{1/q} \geq 0, \quad r\in (0,1).
	\label{Eq:g-sym}
\end{equation}
When $q = 2$, we can also think of this as a rearrangement in the spherical harmonic decomposition of $g$.

\begin{theorem}\label{Thm:scRe}
Let $N\geq 2$, $1 \leq q < \infty$, $g \in L^q(B^N,\RR)$ and $\check g$ be associated to $g$ by \eqref{Eq:g-sym}. We have the following conclusions.

\begin{enumerate}[label = (\roman*)]

\item The map $g \mapsto \check g$ is a $1$-Lipschitz continuous map from $L^q(B^N,\RR)$ into itself:
\[
\|\check g - \check h\|_{L^q(B^N,\RR)} \leq \|g - h\|_{L^q(B^N,\RR)}.
\]
Moreover, $\int_{\mathbb{S}^{N-1}} |\check g(r\theta)|^q\,d\sigma(\theta) = \int_{\mathbb{S}^{N-1}} |g(r\theta)|^q\,d\sigma(\theta)$ for a.e. $r \in (0,1)$.

\item Let $G: [0,\infty) \times [0,\infty) \rightarrow [0,\infty)$ be continuous. If $G$ is convex in the second variable, then
\begin{align*}
\int_{B^N} G(r,|\check g(x)|^q)\,dx \leq \int_{B^N} G(r, |g(x)|^q)\,dx.
\end{align*}
In particular, for any $q < p < \infty$,
\[
\int_{B^N} |\check g|^p\,dx \leq \int_{B^N} |g|^p\,dx.
\]

\item Assume in addition that $g\in W^{1,q}(B^N,\R)$. Then $\check g \in W^{1,q}(B^N,\R)$ and
\begin{equation}\label{Eq:scReI}
\int_{B^N} |\nabla\check g|^q\,dx \leq \int_{B^N}|\nabla g|^q\,dx.
\end{equation}
Equality is attained if and only if $g$ is radially symmetric and $|g| = \check g$ in $(0,1)$.
\end{enumerate}

\end{theorem}

\begin{proof} \underline{Proof of (i):} From the definition of the radial function $\check{g}(x) = \check g(r)$ we have 
\[
\int_{\mathbb{S}^{N-1}}|\check g(r\theta)|^q\,d\sigma(\theta)=\int_{\mathbb{S}^{N-1}} |g(r\theta)|^qd\sigma(\theta) \text{ for a.e. }r\in (0,1),
\]
which implies $\check g \in L^q(B^N)$. Also, by the reverse triangle inequality, we have for $g, h \in L^q(B^N)$ that
\begin{align*}
\|\check g - \check h\|_{L^q(B^N)}^q
	&= |\SN| \int_0^1 |\check g(r) - \check h(r)|^q\,r^{N-1}\,dr \\
	&=  \int_0^1  \big|\|g(r\cdot)\|_{L^q(\SN)} - \|h(r\cdot)\|_{L^q(\SN)}\big|^q\,r^{N-1}\,dr\\
	&\leq \int_0^1  \|g(r\cdot) - h(r\cdot)\|_{L^q(\SN)}^q\,r^{N-1}\,dr
		= \|g - h\|_{L^q(B^N)}^q.
\end{align*}
Therefore $g \mapsto \check g$ is a $1$-Lipschitz continuous map on $L^q(B^N)$.

\medskip
\noindent \underline{Proof of (ii):} By Jensen inequality,
\begin{multline*}
\fint_{\mathbb{S}^{N-1}}G(r, |\check{g}(r\theta)|^q)\,d\sigma(\theta)
	=G\left(r,\fint_{\mathbb{S}^{N-1}}|\check{g}(r\theta)|^q\,d\sigma(\theta)\right)\\
	=G\left(r,\fint_{\mathbb{S}^{N-1}}|g(r\theta)|^q\,d\sigma(\theta)\right)
	\leq \fint_{\mathbb{S}^{N-1}} G(r,|g(r\theta)|^q)\,d\sigma(\theta).
\end{multline*}
Integrating in $r$ gives the second bullet point. In particular, with $G(r,s) = s^{p/q}$ with $p > q$, we see that the $L^p$-norm of $\check g$ is no more than that of $g$.

\medskip
\noindent \underline{Proof of (iii):} Consider first the case $g$ belongs to $\mathcal{C}^\infty(\bar B^N)$, which is a dense subset of $W^{1,q}(B^N)$. For technical reasons, we introduce, for $\mu > 0$,
\[
\check g_\mu(r) = \Big\{\fint_{\mathbb{S}^{N-1}} (g(r\theta)^2 + \mu)^{q/2}\,d\sigma(\theta) \Big\}^{1/q} \geq \mu^{1/2}, \quad r\in (0,1).
\]
Note that $\check g_\mu \rightarrow \check g$ in $L^q(B^N)$ as $\mu \rightarrow 0$. We have, by H\"older's inequality,
\begin{align*}
 |\check g_\mu(r)|^{q-1} |\check g_\mu'(r)| 
 	&\leq  \fint_{\mathbb{S}^{n-1}} (g(r\theta)^2 + \mu)^{(q-1)/2}|\partial_r g(r\theta)|\,d\sigma(\theta)\\
	& \leq |\check g_\mu(r)|^{q-1} \Big\{\fint_{\mathbb{S}^{n-1}}  |\partial_r g(r\theta)|^q\,d\sigma(\theta)\Big\}^{1/q}.
\end{align*}
As $\check g_\mu \geq \mu^{1/2} > 0$, this implies
\[
\fint_{\mathbb{S}^{N-1}} |\nabla \check g_\mu(r\theta)|^q\,d\sigma(\theta) = |\check g_\mu'(r)|^q \leq \fint_{\mathbb{S}^{N-1}} |\partial_r g(r\theta)|^q\,d\sigma(\theta).
\]
Integrating over $r \in (0,1)$ gives 
\[
\int_{B^N} |\nabla \check g_\mu|^q\,dx\leq \int_{B^N} |\partial_r g|^qdx.
\]
This implies $\check g_\mu$ is bounded in $W^{1,q}(B^N)$ and hence converges weakly to $\check g$ in $W^{1,q}(B^N)$ as $\mu \rightarrow 0$. Hence
\begin{equation}
\int_{B^N} |\nabla \check g|^q\,dx\leq \int_{B^N} |\partial_r g|^qdx,
	\label{Eq:scPR}
\end{equation}
which proves \eqref{Eq:scReI} for $g \in \mathcal{C}^\infty(\bar B^N)$.

Suppose now $g \in W^{1,q}(B^N)$. Pick $\{g_{(j)}\}\subset \mathcal{C}^\infty(\bar B^N)$ such that $g_{(j)} \rightarrow g$ in $W^{1,q}(B^N)$. By (i), $\check g_{(j)} \rightarrow \check g$ in $L^q(B^N)$. Also, by \eqref{Eq:scPR},
\begin{equation}
\int_{B^N} |\nabla \check g_{(j)}|^q\,dx\leq \int_{B^N} |\partial_r g_{(j)}|^qdx.
	\label{Eq:scPRm}
\end{equation}
This implies that $\check g_{(j)}$ is bounded in $W^{1,q}(B^N)$ and hence converges weakly in $W^{1,q}(B^N)$ to $\check g$. Sending $j \rightarrow \infty$ we see that \eqref{Eq:scPR} remains valid for $g \in W^{1,q}(B^N)$, which proves \eqref{Eq:scReI}. Moreover, equality holds in \eqref{Eq:scReI} if and only if $|\nabla g| = |\partial_r g|$ a.e., i.e. $g$ is radially symmetric.
\end{proof}

\subsection{A symmetrization of gradient fields and proof of Theorem \ref{main_rearrangement}}\label{ssec:3.2}

Recall the symmetrization $\check v$ for a function $v\in H^1(B^N, \R)$ is given by the  formula \eqref{v-rad}:
\[
\check v(r) = -\int_r^1 \Big\{\fint_{\mathbb{S}^{N-1}} |\nabla v(s\theta)|^2d\sigma(\theta) \Big\}^{1/2}\,ds \leq 0, \quad r\in (0,1).
\]

We will use the following density result.
\begin{lemma}\label{Lem:den}
For $N \geq 2$, the set $\mathcal{S}$ of functions in $\mathcal{C}^\infty(\bar B^N)$ which are constant in a neighborhood of the origin is dense in $H^2(B^N)$. Moreover, if $v\in H^2(B^N)$ verifies $\int_{\SN}v(r\theta)\theta d\sigma(\theta)=0$ for almost every $r\in (0,1)$, then its approximation sequence in $\mathcal{S}$ may be chosen with the same property.
\end{lemma}

\begin{proof}
It is well known that $\mathcal{C}^\infty(\bar B^N)$ is dense in $H^2(B^N)$. Thus, to show that $\mathcal{S}$ is dense in $H^2(B^N)$, we only need to show that a given $v \in \mathcal{C}^\infty(\bar B^N)$ can be approximated by a sequence of functions in $\mathcal{S}$. In the proof, $C$ denotes a constant that can change between lines but depends only on the dimension $N$. Pick a cut-off function $\varphi \in \mathcal{C}^\infty(\RR)$ with $\varphi \equiv 1$ in $(-\infty,1/2]$, $\varphi \equiv 0$ in $[1,\infty)$. For $j \geq 10$ and $x \in B^N$, let
\[
\varphi_{(j)}(x) = \begin{cases}
		\varphi(j|x|) & \text{ if } N \geq 3,\\
		1 - \varphi(\frac{\ln \ln \frac{1}{|x|}}{2\ln\ln j}) & \text{ if } N = 2.
	\end{cases}
\]
Note that $\varphi_{(j)}(x) = 0$ for $|x| \geq \frac{1}{j}$ and $\varphi_{(j)}(x) = 1$ when $|x|$ is small enough. Define
\[
v_{(j)} = v(0) \varphi_{(j)} + v (1 - \varphi_{(j)}) = v - (v - v_0) \varphi_{(j)} \in \mathcal{S}, \quad j \geq 1.
\]

We estimate
\begin{align*}
&|v(x) - v(0)| \leq \|\nabla v\|_{L^\infty(B^N)} |x|,\\
&\|\varphi_{(j)}\|_{L^2(B^N)} \leq Cj^{-N/2},\\
&\|\nabla \varphi_{(j)}\|_{L^2(B^N)} + \|r\nabla^2 \varphi_{(j)}\|_{L^2(B^N)}
	\leq C \omega_N(j) \text{ with } \omega_N(j) = \begin{cases}
	 Cj^{-(N-2)/2} &\text{ if } N \geq 3,\\
	 \frac{C}{(\ln j \ln \ln j)^{1/2}} & \text{ if } N = 2.
	 \end{cases}
\end{align*}
We thus have
\begin{align*}
\|(v - v(0))\varphi_{(j)}\|_{L^2(B^N)} 
	&\leq  Cj^{-N/2}\|v\|_{L^\infty(B^N)} ,\\
\|\nabla[(v - v(0))\varphi_{(j)}]\|_{L^2(B^N)} 
	&\leq  Cj^{-N/2}\|\nabla v\|_{L^\infty(B^N)} + C\omega_N(j)\|v\|_{L^\infty(B^N)},\\
\|\nabla^2[(v - v(0))\varphi_{(j)}]\|_{L^2(B^N)} 
	&\leq  Cj^{-N/2}\|\nabla^2 v\|_{L^\infty(B^N)} +  C\omega_N(j)\|\nabla v\|_{L^\infty(B^N)}.
\end{align*}
Clearly, these estimates imply that $v_{(j)} \rightarrow v$ in $H^2(B^N)$.  We have proved that $\mathcal{S}$ is dense in $H^2(B^N)$

Now suppose $v \in H^2(B^N)$ and $\int_{\SN}v(r\theta)\theta d\sigma(\theta)=0$. Let $v_{(j)} \in \mathcal{S}$ be such that $v_{(j)} \rightarrow v$ in $H^2(B^N)$. Define $\tilde{v}_{(j)}(r\theta)=v_{(j)}(r\theta)-\sum_{k=1}^{N}v_{(j),k}(r)\phi_k(\theta)$ where $v_{(j),k}(r)=\int_{\SN}v_{(j)}(r\theta)\phi_k(\theta)d\sigma(\theta)$. It is clear that $\int_{\SN}\tilde v_{(j)}(r\theta)\theta d\sigma(\theta)=0$. Since $v_{(j)}$ is constant near $0$, $v_{(j),k}$ is supported away from $0$, and so $\tilde v_{(j)} \in \mathcal{S}$. Finally, since the map $w\in H^2(B^N)\mapsto (r\theta\mapsto w_k(r)\phi_k(\theta))$ is continuous in $H^2(B^N)$ and $v_k\equiv 0$ for $k=1,\hdots, N$ we have 
\[
\lim_{j \rightarrow \infty} \|\tilde v_{(j)} - v\|_{H^2(B^N)} \leq \lim_{j\rightarrow \infty} \|v_{(j)} - v\|_{H^2(B^N)} + \lim_{j\rightarrow \infty}  \Big\Vert \sum_{k=1}^{N}v_{(j),k}(r)\phi_k(\theta)\Big\Vert_{H^2(B^N)} = 0.
\]
The proof is complete.
\end{proof}

\begin{proof}[Proof of Theorem \ref{main_rearrangement}]
\underline{Proof of (i):} By Cauchy-Schwarz' inequality,
\begin{align*}
\check v(r)^2 
	&= \Big\{\int_r^1 \Big[\fint_{\mathbb{S}^{N-1}} |\nabla v(s\theta)|^2d\sigma(\theta) \Big]^{1/2}\,ds\Big\}^2\\
	&\leq  \Big\{\int_r^1 s^{1-N}ds\Big\} \Big\{\int_r^1 s^{N-1} \fint_{\mathbb{S}^{N-1}} |\nabla v(s\theta)|^2d\sigma(\theta) \,ds\Big\}.
\end{align*}
Hence $\check v(r)$ is well-defined and finite in $(0,1)$; in fact, $|\check v(r)| \leq C_N r^{-\frac{N-2}{2}} \|\nabla v\|_{L^2(B^N)}$  for $N\geq 3$ (resp. $|\check v(r)|\leq C\sqrt{\log(1/r)}\|\nabla v\|_{L^2(B^2)}$ when $N=2$). In particular, $\check v \in L^2(B^N)$. Moreover, by the definition of $\check{v}$ we have $\int_{\mathbb{S}^{N-1}}|\nabla \check v(r,\theta)|^2\,d\sigma(\theta)=\int_{\mathbb{S}^{N-1}}|\nabla v(r\theta)|^2d\sigma(\theta)$ for a.e. $r\in (0,1)$. As $\check v(1) = 0$, these imply that $\check v \in H_0^1(B^N)$.

As in the proof of (i) in Theorem \ref{Thm:scRe}, the map $\nabla v \mapsto \nabla \check v$ is a $1$-Lipschitz continuous map from $L^2(B^N,\RR^N)$ into itself. By Poincar\'e's inequality, the map $v \mapsto \check v$ is a Lipschitz continuous map from $H^1(B^N)$ into $H_0^1(B^N)$.

\medskip
\noindent \underline{Proof of (ii):} This is similar to that in the proof of Theorem \ref{Thm:scRe} and is omitted.

\medskip
\noindent \underline{Proof of (iii):} By density and (i), it suffices to consider $v \in \mathcal{C}_c^\infty(B^N)$.

Let $A(r) = \fint_{\SN}  |v(r\theta)|^p\,d\sigma(\theta)$.
We have, by H\"older's inequality
\begin{align*}
|A'(r)| 
	&\leq p\fint_{\SN} |v(r\theta)|^{p-1} |\partial_r v(r\theta)|\,d\sigma \leq p\Big\{\fint_{\SN} |v(r\theta)|^{2(p-1)} \,d\sigma\Big\}^{1/2} \check v'(r)\\
	& \leq p\,A(r)^{\frac{p-1}{p}}\,\check v'(r),
\end{align*}
where we have used $2(p-1) \leq p$ when $1 \leq p \leq 2$. 

Fix some $\mu > 0$. Then  $|\frac{d}{dr} (\mu + A(r))^{1/p}| \leq \check v'(r)$. This together with $A(1) = 0 $ (since $v = 0$ on $\partial B^N)$ implies
\[
(\mu + A(r))^{1/p} \leq \mu^{1/p} + \int_r^1 \check v'(r)\,dr = \mu^{1/p} - \check v(r) = \mu^{1/p} + |\check v(r)|.
\]
Sending $\mu \rightarrow 0$, we get the conclusion.

\medskip
\noindent \underline{Proof of (iv):} Without loss of generality, we can assume that $v = 0$ on $\partial B^N$ (since, on $\partial B^N$, $\nabla v(x) = cx$ is normal to $\partial B^N$). Let $(\phi_k)_{k=0}^\infty$ be an orthonormal basis of $L^2(\mathbb{S}^{N-1})$ consisting of eigenfunctions of the Laplace-Beltrami operator on $\SN$ corresponding to eigenvalues $0 = \lambda_0 < N-1 = \lambda_1 = \ldots  = \lambda_N < 2N = \lambda_{N+1} \leq \ldots \rightarrow \infty$.
We decompose
\[
v(r\theta) = \sum_{k=0}^\infty v_k(r) \phi_k(\theta) \text{ where } v_k(r) = \int_{\mathbb{S}^{N-1}} v(r\theta) \phi_k(\theta)\,d\sigma(\theta).
\]
Note that $v_k \in H^2_{\loc}(0,1)$, and
\begin{align}
(\check v')^2 
	&= \sum_{k=0}^\infty \Big[(v_k' )^2 + \frac{\lambda_k}{r^2} v_k^2 \Big],
	\label{Eq:barv}\\
\int_{B^N} (\Delta v)^2\,dx
	&= \sum_{k=0}^\infty \int_0^1 r^{N-1} \Big(v_k'' + \frac{N-1}{r} v_k' - \frac{\lambda_k}{r^2} v_k\Big)^2\,dr.
	\label{Eq:DelvL2}
\end{align}
Note also that our hypotheses give in the case $N \in \{3,4\}$ that $v_1 = \ldots = v_N = 0$.

We first prove inequality \eqref{Eq:RearrIneql} when $v$ belongs to the set $\mathcal{S}$ defined in Lemma \ref{Lem:den}. Then $v_0 \in \mathcal{C}^\infty([0,1])$, $v_0$ is constant near $0$, $v_k \in   \mathcal{C}_c^\infty((0,1])$ for $k \geq 1$, 
\[
v_0(1) = 0, v_0'(1) = c \text{ and } v_k(1) = v_k'(1) = 0 \text{ for } k \geq 1.
\]
This implies $\nabla \check v(x) = |c| x$ on $\partial B^N$ (recall that, by definition, $\check v' \geq 0$ in $(0,1)$). Also,
\begin{align*}
\int_0^1 r^{N-2} v_k'' v_k'\,dr &=  -\frac{N-2}{2} \int_0^1 r^{N-3} (v_k')^2\,dr + \begin{cases}
	\frac{c^2}{2}& \text{ if } k = 0,\\
	0 & \text{ if } k \geq 1,
\end{cases}
\\
\int_0^1 r^{N-4} v_k' v_k\,dr &=  -\frac{N-4}{2} \int_0^1 r^{N-5} v_k^2\,dr
	\text{ for } k \geq 1 , \\
\int_0^1 r^{N-3} v_k'' v_k\,dr &=  \int_0^1 \Big[-r^{N-3} (v_k')^2 + \frac{(N-3)(N-4)}{2}  r^{N-5} v_k^2\Big]\,dr\text{ for } k \geq 1.
\end{align*}
Inserting the above identities in \eqref{Eq:DelvL2}, we obtain
\begin{equation}
\int_{B^N} (\Delta v)^2\,dx
	= (N - 1)c^2 + \sum_{k=0}^\infty \int_0^1 r^{N-1} \Big[(v_k'')^2 + \frac{2\lambda_k + N-1}{r^2} (v_k')^2 + \frac{\lambda_k(\lambda_k + 2(N-4))}{r^4} v_k^2\Big]\,dr.
	\label{Eq:DvE}
\end{equation}

Next, note that, when $v \in \mathcal{S}$, the right hand side of \eqref{Eq:barv} is a smooth non-negative function and so $\check v'$ is Lipschitz continuous. Applying \eqref{Eq:DvE} to $\check v$, we get
\begin{align}
\int_{B^N} (\Delta \check v)^2\,dx
	&=  (N - 1)c^2 + \int_0^1 r^{N-1} \Big[(\check v'')^2 + \frac{N-1}{r^2} (\check v')^2\Big]\,dr.
	\label{Eq:DcvE}
\end{align}
To continue, we need to estimate $\check v''$. For technical reasons, we consider for $\mu > 0$ a regularized version of $\check v$:
\[
\check v_\mu' = \Big\{\mu + \sum_{k=0}^\infty \Big[(v_k' )^2 + \frac{\lambda_k}{r^2} v_k^2 \Big]\Big\}^{1/2} \geq \mu^{1/2}.
\]
Clearly $\check v_\mu'$ is smooth and $\check v_\mu' \rightarrow \check v'$ pointwise in $(0,1)$ as $\mu \rightarrow 0$. Now, for some $t_k \in \R$ to be chosen later, we have by \eqref{Eq:barv} that
\begin{align*}
|\check v_\mu'| |\check v_\mu''| 
	&= \Big|\sum_{k=0}^\infty \Big[v_k' v_k'' + \frac{\lambda_k}{r^2} v_k v_k' - \frac{\lambda_k}{r^3} v_k^2 \Big]\Big|\\
	&\leq |v_0'| |v_0''| +  \Big|\sum_{k=1}^\infty \Big[v_k' (v_k'' + \frac{t_k}{r^2} v_k) + \frac{1}{r} v_k( \frac{\lambda_k - t_k}{  r} v_k' - \frac{\lambda_k }{r^2} v_k)\Big]\Big|\\
	&\leq |v_0'| |v_0''| +  \sum_{k=1}^\infty \Big[(v_k' )^2 + \frac{\lambda_k}{r^2} v_k^2 \Big]^{1/2}\Big[(v_k'' + \frac{t_k}{r^2} v_k)^2 + 
 \frac{1}{\lambda_k} ( \frac{\lambda_k - t_k}{  r} v_k' - \frac{\lambda_k }{r^2} v_k)^2 \Big]^{1/2}\\
	&\leq |\check v_\mu'| \Big\{|v_0''|^2 + \sum_{k=1}^\infty \Big[(v_k'' + \frac{t_k}{r^2} v_k)^2 + \frac{1}{\lambda_k} ( \frac{\lambda_k - t_k}{  r} v_k' - \frac{\lambda_k }{r^2} v_k)^2 \Big]\Big\}^{1/2}.
\end{align*}
Since $\check v_\mu' \geq \mu^{1/2} > 0$, this implies
\[
|\check v_\mu''| 
	\leq \Big\{|v_0''|^2 + \sum_{k=1}^\infty \Big[(v_k'' + \frac{t_k}{r^2} v_k)^2 + \frac{1}{\lambda_k} ( \frac{\lambda_k - t_k}{  r} v_k' - \frac{\lambda_k }{r^2} v_k)^2 \Big]\Big\}^{1/2}.
\]
This implies that $\{\check v_\mu'\}$ is bounded in $W^{1,\infty}((0,1))$ and converges weakly* in $W^{1,\infty}((0,1))$ to $\check v'$ as $\mu \rightarrow 0$ (since $\check v_\mu' \rightarrow \check v'$ pointwise), and
\[
|\check v''| 
	\leq \Big\{|v_0''|^2 + \sum_{k=1}^\infty \Big[(v_k'' + \frac{t_k}{r^2} v_k)^2 + \frac{1}{\lambda_k} ( \frac{\lambda_k - t_k}{  r} v_k' - \frac{\lambda_k }{r^2} v_k)^2 \Big]\Big\}^{1/2}.
\]
Returning to \eqref{Eq:DcvE}, we get
\begin{align*}
&\int_{B^N} (\Delta \check v)^2\,dx
	\leq   (N - 1)c^2 + \int_0^1 r^{N-1}\Big[(v_0'')^2 + \frac{N-1}{r^2} (v_0')^2\Big]\,dr \\
		&\quad+ \sum_{k=1}^\infty\int_0^1 r^{N-1} \Big[(v_k'' + \frac{t_k}{r^2} v_k)^2 + \frac{1}{\lambda_k} ( \frac{\lambda_k - t_k}{  r} v_k' - \frac{\lambda_k }{r^2} v_k)^2	+ \frac{N-1}{r^2} (v_k')^2 + \frac{(N-1)\lambda_k}{r^4} v_k^2\Big]\,dr\\
	&=  (N - 1)c^2 + \int_0^1 r^{N-1}\Big[(v_0'')^2 + \frac{N-1}{r^2} (v_0')^2\Big]\,dr\\
		&\quad + \sum_{k=1}^\infty\int_0^1 r^{N-1} \Big[(v_k'')^2 + \frac{\lambda_k^{-1} (\lambda_k - t_k)^2 - 2t_k + N - 1}{r^2} (v_k')^2\\
			&\qquad + \frac{2\lambda_k(N-2) + t_k^2 + t_k(N-4)^2}{r^4} v_k^2  \Big]\,dr.
\end{align*}
Recalling \eqref{Eq:DvE}, we get
\begin{align}
\int_{B^N} (\Delta   v)^2\,dx - \int_{B^N} (\Delta \check v)^2\,dx
	&\geq \sum_{k=1}^\infty\int_0^1 r^{N-1} \Big[  \frac{ \lambda_k  - \lambda_k^{-1} t_k^2 + 4t_k }{r^2} (v_k')^2\nonumber\\
		&\qquad + \frac{\lambda_k^2 - 4\lambda_k  - t_k^2 - t_k(N-4)^2}{r^4} v_k^2  \Big]\,dr.
		\label{Eq:X1}
\end{align}

Case 1: If $N \geq 5$, we choose $t_k = 0$, and using the sharp Hardy inequality $\int_0^1 r^{N-3} (v_k')^2\,dr \geq \frac{(N-4)^2}{4} \int_0^1 r^{N-5} v_k^2\,dr$ to obtain from \eqref{Eq:X1} the inequality
\begin{equation}\label{Eq:PreRearr}
\int_{B^N} (\Delta   v)^2\,dx \geq \int_{B^N} (\Delta \check v)^2\,dx + \sum_{k=1}^\infty \lambda_k s_k \int_0^1 r^{N-5}   v_k^2  \,dr,
\end{equation}
where, for $k \geq 1$,
\[
s_k = \lambda_k + \frac{(N-4)^2}{4} - 4 > 0  \quad (\text{since } \lambda_k \geq N - 1 \geq 4).
\]
Inequality \eqref{Eq:RearrIneql} thus follows.

Case 2: If $N \in \{2,3,4\}$, recall that our hypotheses give $v_1 = \ldots = v_{N} = 0$. We choose $t_k = (2-\sqrt{5})\lambda_k$ in \eqref{Eq:X1} so that the term involving $v_k'$ vanishes, and arrive again at \eqref{Eq:PreRearr} but with
\[
s_k = \begin{cases}
	0 & \text{ if } 1 \leq k \leq N,\\
	(\sqrt{5}-2)(4\lambda_k+(N-4)^2) - 4 &\text{ if }  k \geq N + 1.
	\end{cases}
\]
As $\lambda_k \geq 2N$ for $k \geq N+1$, we have 
\begin{align*}
s_k 
	&\geq (\sqrt{5}-2) (N^2 + 16) - 4 \geq 20 \sqrt{5} - 44  > 0 \text{ for } N \in \{2,3,4\}, k \geq N+1.
\end{align*}
Inequality \eqref{Eq:RearrIneql} thus follows from \eqref{Eq:PreRearr}.

Consider now the general case $v \in H^2(B^N)$. By Lemma \ref{Lem:den} we can select $\{v_{(j)}\} \subset \mathcal{S}$ such that $v_{(j)} \rightarrow v$ in $H^2(B^N)$ as $j \rightarrow \infty$. Moreover, in case $N \in \{2,3,4\}$, it holds also that $\int_{\SN} v_{(j)}(r\theta)\theta\,d\sigma(\theta) = 0$. By Fubini's theorem, after passing to a subsequence, we have $(v_{(j)})_k \rightarrow v_k$ a.e. in $(0,1)$ for the spherical harmonic coefficients of $v_{(j)}$ and $v$. Also, by (i), $\nabla\check v_{(j)} \rightarrow \nabla \check v$ in $L^2(B^N,\R^N)$. Since $v_{(j)} \in \mathcal{S}$, we have by \eqref{Eq:PreRearr} 
\begin{equation}
\int_{B^N} (\Delta   v_{(j)})^2\,dx \geq  \int_{B^N} (\Delta \check v_{(j)})^2\,dx + \sum_{k=1}^\infty \lambda_k s_k \int_0^1 r^{N-5}   (v_{(j)})_k^2  \,dr.
	\label{Eq:PRm}
\end{equation}
This implies that $\{\check v_{(j)}\}$ is bounded in $H^2(B^N)$. As $\nabla\check v_{(j)} \rightarrow \nabla \check v$ in $L^2(B^N,\R^N)$, this implies $\Delta \check v_{(j)}$ converges weakly in $L^2(B^N)$ to $\Delta \check v$; in particular, $\check v \in H^2(B^N)$. Sending $j \rightarrow \infty$ in \eqref{Eq:PRm}, using the convergence of $v_{(j)}$ to $v$ in $H^2(B^N)$ on the left hand side, the weak convergence of $\Delta\check v_{(j)}$ to $\Delta\check v$ in $L^2(B^N)$ and Fatou's lemma for the infinite sum on the right hand side, we see that \eqref{Eq:PreRearr} remains valid for $v \in H^2(B^N)$. This proves \eqref{Eq:RearrIneql} for $v \in H^2(B^N)$. Also, equality occurs in \eqref{Eq:RearrIneql} if and only if $v_k = 0$ for all $k \geq 1$, meaning $v$ is radially symmetric and $|v'| = \check v'$ in $(0,1)$.
\end{proof}

\subsection{Second proof of Theorems \ref{thm:AG}, \ref{thm2} and \ref{thm1} in dimension $N\geq 5$}

\begin{proof}[Second proof of Theorem \ref{thm:AG} in dimension $N\geq 5$]
As $s \mapsto W(1 - s)$ is convex, we deduce from Theorem \ref{main_rearrangement} that
\[
E_\eps^{GL}[\nabla u] \geq E_\eps^{GL}[\nabla \check u] \text{ for all } \nabla u \in \mcA^{GL},
\]
where equality holds if and only if $u$ is radially symmetric. In particular, if $\nabla u \in \mcA^{GL}$ is a minimizer of $E_\eps^{GL}$ among gradient field configurations in $\mcA^{GL}$, then so is $\nabla \check u$ with $E_\eps^{GL}[\nabla u] = E_\eps^{GL}[\nabla \check u]$ and hence $u$ is radially symmetric. The conclusion then follows from \cite[Theorem 2.1]{IN24-IHP} on the uniqueness of radially symmetric critical point of $E_{\eps}^{GL}$ in $\mcA^{GL}$.
\end{proof}

\begin{proof}[Second proof of Theorem \ref{thm2} in dimension $N\geq 5$]
Observe that if $(\nabla m, M_{N+1}) \in \mcA^{MM}$ and if $\check m$ denotes the symmetrization of $m$ by \eqref{v-rad} and $\check{M}_{N+1}$ denotes the symmetrization of $M_{N+1}$ by \eqref{Eq:g-sym}, then $(\nabla \check{m}, \check{M}_{N+1}) \in \mcA^{MM}$ because
\[
|\nabla \check{m}|^2(r) + \check{M}_{N+1}(r)^2 = \fint_{\SN} \big(|\nabla m|^2(r\theta) + M_{N+1}^2(r\theta)\big)\,d\sigma(\theta) = 1.
\]
Thus, by Theorems \ref{main_rearrangement} and \ref{Thm:scRe}, if $(\nabla m, M_{N+1}) \in \mcA^{MM}$ is a minimizer of $E_\eta^{MM}$ in $\mcA^{MM}$, the $(\nabla \check m, \check M_{N+1})$  is also a minimizer of $E_\eta^{MM}$ in $\mcA^{MM}$ and $(\nabla m, M_{N+1})$ is radially symmetric. The conclusion then follows from \cite[Theorem 2.6]{IN24-IHP} on the classification of radially symmetric minimizers of $E_{\eta}^{MM}$.
\end{proof}

\begin{proof}[Second proof of Theorem \ref{thm1} in dimension $N\geq 5$]
Let $U = (\nabla v, g)\in \mcA$ be a minimizer of $E_{\eps,\eta}$ in $\mcA$. Define the symmetrization $\check v$ and $\check g$ of $v$ and $g$ as in the previous two sections with $q = 2$, and let $\check U = (\nabla \check v, \check g)$. By Theorems \ref{main_rearrangement} and \ref{Thm:scRe}, we have
\begin{align*}
\int_{\SN} |\nabla \check v(r\theta)|^2\,d\sigma(\theta) 
	&= \int_{\SN} |\nabla v(r\theta)|^2\,d\sigma(\theta) \text{ for a.e. } r \in (0,1),\\
\int_{\SN} \check g(r\theta)^2\,d\sigma(\theta) 
	&= \int_{\SN} g(r\theta)^2\,d\sigma(\theta) \text{ for a.e. } r \in (0,1),\\
\int_{B^N} \tilde W(\check g^2)\,dx
	&\leq \int_{\SN} \tilde W(g^2)\,dx,\\
\int_{B^N} (\Delta \check{v})^2\,dx
	&\leq \int_{B^N} (\Delta v)^2\,dx.
\end{align*}
The first two identities and the convexity of $W$ give
\begin{multline*}
\fint_{\SN} W(1 - |\check U|^2)(r\theta)\,d\sigma(\theta)
	= W\Big(\fint_{\SN}(1 - |\check U|^2)(r\theta)\,d\sigma(\theta)\Big)\\
	= W\Big(\fint_{\SN}(1 - |U|^2)(r\theta)\,d\sigma(\theta)\Big)
	\leq \fint_{\SN} W(1 - |U|^2)(r\theta)\,d\sigma(\theta).
\end{multline*}
These estimates together implies that $E_{\eps,\eta}[\check U] \leq E_{\eps,\eta}[U]$ and so $\check U$ is also a minimizer of $E_{\eps,\eta}$ in $\mcA$ with $E_{\eps,\eta}[\check U] = E_{\eps,\eta}[U]$. Returning to the equality cases in Theorems \ref{main_rearrangement} and \ref{Thm:scRe}, we have that $v$ and $g$ are radially symmetric, i.e. $U$ is a radially symmetric minimizer of $E_{\eps,\eta}$. The conclusion follows from \cite[Theorem 2.4]{IN24-IHP} on the classification of radially symmetric minimizer of $E_{\eps,\eta}$.
\end{proof}

\subsection{Symmetry for solutions to a nonlinear eigenvalue problem}\label{ssec:RApp}

For $d > 0$, $1 \leq p < 2$ and $\lambda \in \RR$, consider the energy functional
\[
J[v] = \frac{1}{2}\|\Delta v\|_{L^2(B^N)}^2 - \frac{\lambda}{2}\|v\|_{L^2(B^N)}^2
\]
on the set
\[
S_{p,d} = \big\{v \in H^2_0(B^N): \|v\|_{L^p(B^N)} = d\big\}.
\]
Let  $\lambda_1(\Delta^2)$ denote the first eigenvalue of the bi-Laplacian in $H_0^2(B^N)$. When $\lambda <  \lambda_1(\Delta^2)$, after adjusting by a scaling factor to remove the Lagrange multiplier, minimizers of $J$ on $S_{p,d}$ correspond to solutions of the elliptic problem \eqref{Eq:BHE}:
\[
\begin{cases}
\Delta^2 v  = \lambda v + |v|^{p-2} v \text{ on } B^N,\\
v = \partial_r v = 0 \text{ on } \partial B^N.
\end{cases}
\]
(For $\lambda \geq \lambda_1(\Delta^2)$, the partial differential equation is different, namely
\[
\Delta^2 v  = \begin{cases} 
	\lambda_1(\Delta^2) v  &\text{ if } \lambda = \lambda_1(\Delta^2) ,\\
	\lambda v - |v|^{p-2} v  &\text{ if } \lambda > \lambda_1(\Delta^2) ,
	\end{cases}
\]
and we do not consider these cases here for simplicity.)

Problem \eqref{Eq:BHE} has been studied by many authors and a summary of known results would go beyond the scope of the present paper. We refer the reader to e.g. \cite{BernisAzoreroPeral96, Bonheureetal19-TrAMS, EdmundsFortunatoJannelli90, FrankLenzmannSilvestre16-CPAM, GazzolaGrunauSquassina03, lenzmann21} and the references therein.

We prove:

\begin{corollary}\label{Cor:6}
Let $N \geq 5$ and $1 \leq p < 2$. For $\lambda < \lambda_1(\Delta^2)$, minimizers of $J$ over $S_{p,d}$ are radially symmetric, do not change sign and are either radially non-decreasing or radially non-increasing.
\end{corollary}

\begin{proof}
Note that as $\lambda <  \lambda_1(\Delta^2)$, $J$ is coercive on $H_0^2(B^N)$. By the compactness embedding theorem, $J$ has a minimizer over $S_{p,d}$. 

Let $v$ be a minimizer of $J$ over $S_{p,d}$; in particular, $J[v] \geq 0$. By Theorem \ref{main_rearrangement}, we have
\begin{equation}
\begin{cases}
\|\Delta\check v\|_{L^2(B^N)} \leq \|\Delta v\|_{L^2(B^N)},\\
\|\check v\|_{L^p(B^N)} \geq \| v\|_{L^p(B^N)}  = d, \\
\|\check v\|_{L^2(B^N)} \geq \| v\|_{L^2(B^N)}.
\end{cases}
	\label{Eq:C6vv}
\end{equation}
Let
\[
\bar v = \mu \check v \text{ where } \mu = d\|\check v\|_{L^p(B^N)}^{-1} \stackrel{\eqref{Eq:C6vv}}{\leq} 1
\]
so that $\bar v \in S_{p,d}$. We compute, keeping in mind that $\mu \leq 1$,
\[
J[\bar v] = \mu^2 J[\check v] \stackrel{\eqref{Eq:C6vv}}{\leq} \mu^2 J[v] \leq J[v],
\]
where for the last inequality we use the fact that $J[v] \geq 0$. It follows that $\bar v$ is also a minimizer of $J$ over $S_{p,d}$, which in turn implies $J[\check v] = J[v]$ and all inequalities in \eqref{Eq:C6vv} are saturated. Appealing to the equality case in Theorem \ref{main_rearrangement}(iv), we see that $v$ is radially symmetric and $\check v' = |v'|$. 

It remains to prove that $v$ and $\partial_r v$ do not change sign. Indeed, we have
\[
|v(r)| = |v(r) - v(1)| = \Big|\int_r^1 v'(s)\,ds\Big| \leq \int_r^1 |v'(s)|\,ds = \int_r^1 \check v'(s)\,ds = |\check v(r)|.
\]
As $\|\check v\|_{L^2(B^N)} = \|v\|_{L^2(B^N)}$, it follows that equality is attained in the above inequality, i.e. $v'$ does not change sign. As $v(1) = 0$, it follows also that $v$ does not change sign.
\end{proof}

\appendix
\section{The negativity of $F_\eps$ in dimension $N \in \{2,3\}$} \label{app:F<0}

We now give the proof of Proposition \ref{Prop:F<0} on the negativity of $F_\eps$ in dimension $N \in\{2,3\}$. 
\begin{proof}[Proof of Proposition \ref{Prop:F<0}]
We follow ideas from e.g. the proof of \cite[Lemma 2.3]{IN24-IHP}, \cite[Proposition 4.1]{INSZ_CRAS}, \cite[Theorem 1.7]{INSZ_AnnIHP}. The main task is to show that there exists $v \in \mathcal{C}_c^2(B^N \setminus \{0\})$ such that\footnote{Note that $\mathcal{C}_c^2(B^N \setminus \{0\})$ is not a dense subspace of $H_0^2(B^N)$ in dimension $N \in \{2,3\}$, hence the existence of such $v$ does not follow immediately from the sharpness of the Hardy inequality \eqref{eq_HardyRellich}.}
\begin{equation}
F_*[\nabla v] := \int_{B^N} \Big[(\Delta v)^2\, - \frac{N-1}{r^2} |\nabla v|^2\Big]\,dx < 0.
	\label{Eq:PTrial}
\end{equation}
Supposing for the moment that such a $v$ has been found, we proceed to show that $F_\eps[\nabla v] < 0$ for this particular $v$ and for sufficiently small $\eps > 0$. Indeed, using the Hardy decomposition Lemma \ref{Lem:HardyD} with the decomposition $\nabla v = f_\eps \frac{\nabla v}{f_\eps}$, noting that $\Delta f_\eps = \frac{N-1}{r^2} f_\eps - \frac{1}{\eps^2} W'(1 - f_\eps^2) f_\eps$, we find
\[
F_\eps[\nabla v] = \int_{B^N} f_\eps^2\Big[\Big|\nabla\big(\frac{\nabla v}{f_\eps}\big)\Big|^2 - \frac{N - 1}{r^2} \frac{|\nabla v|^2}{f_\eps^2}\Big]\,dx.
\]
Since $f_\eps \rightarrow 1$ in $\mathcal{C}^1_{\rm loc}(B^N \setminus \{0\})$ and $v \in \mathcal{C}_c^2(B^N \setminus \{0\}$, we deduce that
\[
\lim_{\eps\rightarrow 0} F_\eps[\nabla v] = \int_{B^N} \Big[|\nabla^2 v|^2 - \frac{N - 1}{r^2} |\nabla v|^2\Big]\,dx = F_*[\nabla v] < 0,
\]
which gives the conclusion.

It remains to find $v \in \mathcal{C}_c^2(B^N\setminus\{0\})$ satisfying \eqref{Eq:PTrial}. The proof of Theorem \ref{main_rearrangement} suggests the ansatz
\[
v(x) = a(r)\frac{x_1}{r}.
\]
We are thus led to searching for $a \in \mathcal{C}_c^2((0,1))$ such that
\[
\int_0^1 r^{N-1}\Big[(a'')^2 + \frac{2(N-1)}{r^2} (a')^2 + \frac{2(N-1)(N-4)}{r^4} a^2\Big]\,dr < 0.
\]
We decompose $a(r) = r^{-\frac{N-4}{2}}b(r)$ and compute
\begin{align*}
\int_0^1 r^{N-1}(a'')^2\,dr
	&= \int_0^1 r^3 \Big(b'' - \frac{N-4}{r} b' + \frac{(N-2)(N-4)}{4r^2}b\Big)^2\,dr\\
	&= \int_0^1 \Big(r^3 (b'')^2 + \frac{(N-2)(N-4)}{2} r (b')^2 + \frac{(N-2)^2(N-4)^2}{16r} b^2\Big)\,dr,\\
\int_0^1 r^{N-3}(a')^2\,dr
	&= \int_0^1 r\Big(b' - \frac{N-4}{2r} b\Big)^2\,dr\\
	&= \int_0^1 \Big(r (b')^2 + \frac{(N-4)^2}{4r} b^2\Big)\,dr
\end{align*}
We thus need to find $b \in \mathcal{C}_c^2((0,1))$ such that
\begin{equation}
\int_0^1 \Big( r^3 (b'')^2 + \frac{N^2 - 2N + 4}{2} r (b')^2 + \frac{(N-4)(N^3 + 12N - 16)}{16r} b^2\Big)\,dr < 0.
	\label{Eq:Trial}
\end{equation}
To this end, we fix a cut-off function $\varphi \in \mathcal{C}_c^\infty([0,\infty))$ with $\varphi \equiv 1$ in $[0,1/4]$, $\varphi \equiv 0$ in $[1/2,\infty)$. For $j > 20$ large to be fixed, we let
\[
b(r) = \begin{cases}
	\varphi(r) & \text{ if } r \geq 1/8,\\
	\varphi\big(\frac{\ln \ln \frac{1}{r}}{4\ln \ln j}\big) & \text{ if } 0 < r < 1/8.
	\end{cases}
\]
Then, for some constant $C$ independent of $j$, we have
\begin{align*}
&\int_0^1 \frac{1}{r}b^2\,dr
	 \geq \int_{1/j}^{1/4} \frac{dr}{r} =  \ln \frac{j}{4},\\
&\int_0^1 \Big(r^3 (b'')^2 + r (b')^2 \Big)\,dr
	\leq C.
\end{align*}
Therefore, as $(N-4)(N^3 + 12N - 16) < 0$ for $N \in \{2,3\}$, we can select a sufficiently large $j$ so that \eqref{Eq:Trial} is satisfied.
\end{proof}

\def\cprime{$'$}

\end{document}